\documentclass[a4paper,11pt,twoside]{amsart}
\usepackage[latin1]{inputenc}
\usepackage{amsmath, amsfonts, amssymb,amsthm}
\usepackage[mathscr]{eucal}
\usepackage[pdftex]{graphicx}

\usepackage[T1]{fontenc}
\usepackage[sc]{mathpazo}
\usepackage[all]{xy}

\usepackage{enumerate}

\newcommand{\N}{\mathbb{N}}
\newcommand{\Z}{\mathbb{Z}}
\newcommand{\R}{\mathbb{R}}
\newcommand{\C}{\mathbb{C}}
\newcommand{\s}{\mathbb{S}}
\newcommand{\h}{\mathbb{H}}

\newcommand{\df}{\mathrm{d}}

\newcommand{\prodesc}[2]{\left\langle #1, #2 \right \rangle}
\newcommand{\vprodesc}[2]{\langle\!\langle #1, #2 \rangle\!\rangle}
\newcommand{\abs}[1]{\left\lvert #1 \right\rvert}

\newtheorem{theorem}{Theorem}
\newtheorem*{theorem*}{Theorem}
\newtheorem{proposition}{Proposition}
\newtheorem{corollary}{Corollary}
\newtheorem{lemma}{Lemma}

\theoremstyle{definition}
  \newtheorem{definition}{Definition}

\theoremstyle{remark}
  \newtheorem{remark}{Remark}

\numberwithin{equation}{section}

\title{Minimal surfaces in $\s^2  \times \s^2$}

\author{Francisco Torralbo}
\address{Departamento de Geometr\'{\i}a  y Topolog\'{\i}a. Universidad de Granada. 18071 Granada, SPAIN}
\email{ftorralbo@ugr.es}
\thanks{Research partially supported by a MCyT-Feder research project MTM2011-22547 and Junta Andalucía Grants P09-FQM-5088 and P09-FQM-4496.}

\author{Francisco Urbano}
\address{Departamento de Geometr\'{\i}a  y Topolog\'{\i}a.
Universidad de Granada. 18071 Granada, SPAIN}
\email{furbano@ugr.es}



\begin{document}
\maketitle

\begin{abstract}
A general study of minimal surfaces of the Riemannian product of two spheres $\s^2\times\s^2$ is tackled. We stablish a local correspondence between (non-complex) minimal surfaces of $\s^2 \times \s^2$ and certain pair of minimal surfaces of the sphere $\mathbb{S}^3$. This correspondence also allows us to link minimal surfaces in $\mathbb{S}^3$ and in the Riemannian product $\s^2 \times \R$. Some rigidity results for compact minimal surfaces are also obtained.
\end{abstract}

\section{Introduction}
The theory of minimal surfaces in $4$-dimensional Riemannian manifolds and more particularly in Einstein-Kähler surfaces is an interesting topic in submanifold theory which has been treated from different points of view (see~\cite{B, CU, EGT, MW, MU,W}). Besides the complex projective plane, there is only one more Hermitian symmetric space of compact type and complex dimension 2: the Einstein-Kähler surface $\s^2\times\s^2$. In \cite{CU} its minimal Lagrangian surfaces and in \cite{TU1} its surfaces with parallel mean curvature vector were studied in depth. Also, in \cite{TU2} was proved that the stable minimal compact surfaces of $\s^2\times\s^2$ are the complex ones.

In this paper the authors broach the general study of minimal surfaces of $\s^2\times\s^2$,
including those which are not full in $\s^2\times\s^2$, i.e., minimal surfaces of $\s^2\times\s^1$ (or minimal surfaces of its universal covering $\s^2\times\R$). Perhaps, the most interesting result in the paper is that, roughly speaking, there is a local correspondence between (non-complex) minimal surfaces of $\s^2\times\s^2$ and certain pair of minimal surfaces of the $3$-dimensional sphere $\s^3$. This correspondence is made through the Gauss map of the pair of minimal surfaces of the $3$-sphere.

The paper is organized as follows. Section~\ref{sec:preliminares} describes the most important geometric properties of the ambient manifold $\s^2\times\s^2$, including its two structures of Einstein-Kähler manifold, as well as its identification with the Grasmannian manifold of oriented  $2$-planes of $\R^4$. This identification will be crucial to understand the aforementioned correspondence. The two complex structures of $\s^2\times\s^2$ allow to define on any minimal surface $\Sigma$ of $\s^2\times\s^2$ two Kähler functions $C_1,C_2:\Sigma\rightarrow [-1,1]$ which will play an important role along the paper. We finish Section~\ref{sec:preliminares} by exposing in detail the most regular examples of minimal surfaces of $\s^2\times\s^2$: slices, the diagonal $\textbf{D}$, the Clifford torus $\textbf{T}$ and certain complex tori constructed with the Weierstrass $\wp$-function.

In Section~\ref{sec:structure-equation} we study the Frenet equations and the compatibility equations of a minimal surface $\Sigma$ of $\s^2\times\s^2$. We also define the holomorphic Hopf differential associated to such $\Sigma$ and we show that complex surfaces are just those with vanishing Hopf differential. In Lemma~\ref{lm:properties-fundamental-data} we obtain certain differential equations that must be satisfied by the Kähler functions of $\Sigma$ which allow to compute, when $\Sigma$ is compact, the number of complex points of $\Sigma$ in terms of the degrees of the two components of the immersion into $\s^2\times\s^2$ and the Euler numbers of $\Sigma$ and its normal bundle.

In Section~\ref{sec:local-results} we obtain certain local results which are important to understand the rest of the paper. So, in Proposition~\ref{prop:totallygeodesic}, slices, the diagonal $\textbf{D}$ and the Clifford torus $\textbf{T}$ are locally characterized respectively as the only minimal surfaces of $\s^2\times\s^2$ which are totally geodesic, with constant Kähler functions, or with constant Gauss and normal curvatures. Also, in Proposition~\ref{prop:curvaturanormalnula}, we characterize the non-full minimal surfaces of $\s^2\times\s^2$ as the only ones with $C_1^2=C_2^2$.
We finish this section associating to any minimal surface of $\s^2\times\s^2$ without complex points a pair of smooth functions which satisfy a sinh-Gordon equation (cf.\ Proposition~\ref{prop:sinh-Gordon}). This fact has a converse, because any two solutions of the sinh-Gordon equation have associated a $1$-parameter family of isometric minimal immersions in $\s^2\times\s^2$ (cf. Theorem~\ref{thm:sinh-Gordon->minima-s2xs2}).

In Section~\ref{sec:gauss-map} we study the local correspondence mentioned before. So, in Theorem~\ref{thm:aplicacion-gauss-minima-s3} we give a method to construct minimal surfaces of $\s^2\times\s^2$. In fact, given two minimal immersions of an oriented surface $\Sigma$ into the sphere $\s^3$ with conformal induced metrics and the same Hopf differentials, we construct, using their Gauss maps, a minimal immersion of $\Sigma$ into $\s^2\times\s^2$ whose induced metric is conformal to the above ones. In Theorem~\ref{thm:minima-s2xs2-localmente-aplicacion-gauss} we prove a local converse of this fact, by stablishing that any simply connected minimal surface of $\s^2\times\s^2$ without complex points comes from a pair of minimal surfaces of $\s^3$ in the above way. It is interesting to remark that, as a consequence, minimal surfaces of $\s^2\times\R$ locally come from minimal surfaces of $\s^3$ through the Gauss map (see Corollary~\ref{cor:gauss-map-minimal-S2xR}).

Finally, Section~\ref{sec:compact} is devoted to the study of {\it compact} minimal surfaces of $\s^2\times\s^2$. In Proposition~\ref{prop:genus} we study some interesting properties of its compact complex surfaces related with their areas. As the area of any compact complex surface of $\s^2\times\s^2$ is a integer multiple of $4\pi$, we classify those with area less than $16\pi$ and with area $16\pi$ under the assumptions of embedeness. Among surfaces that appear in the last case, there are the Weierstrass tori mentioned before. The last two theorems of the paper are global rigidity results involving the Gauss and normal curvatures and can be summarized as follows:
\begin{quote}
\begin{enumerate}
{\it \item The slices and the diagonal $\textbf{D}$ are the only compact minimal surfaces of $\s^2\times\s^2$ either with positive constant Gauss curvature or whose Gauss curvature satisfies $K \geq 1/2$.
\item The Clifford torus $\textbf{T}$ is the only flat compact minimal surface of $\s^2\times\s^2$ and  $\textbf{T}$ and $\textbf{D}$ are the only compact minimal surfaces satisfying $0\leq K\leq 1/2$.

\item Any minimal torus of $\s^2\times\s^2$ with normal curvature $K^{\perp}=0$ lies in a totally geodesic hypersurface of $\s^2\times\s^2$.

\item For any orientable compact minimal surface of $\s^2\times\s^2$, the functions $K\pm K^{\perp}$ cannot be everywhere negative and $K\pm K^{\perp}\geq 0$ if and only if the surface is either complex or Lagrangian.}
\end{enumerate}
\end{quote}

\section{Preliminaries}\label{sec:preliminares}
Let $\s^2$ be the $2$-dimensional sphere, $\prodesc{\,}{\,}$ its standard metric, $J$ its complex structure and $\omega$ its Kähler $2$-form, i.e. $\omega(\,,\, ) = \prodesc{J\,}{\,}$.

We endow $\s^2\times\s^2$ with the product metric (also denote by $\prodesc{\,}{\,}$) and the complex structures
\[
J_1=(J,J) \quad \text{and}\quad  J_2=(J,-J),
\]
which define two structures of Kähler surface on $\s^2\times\s^2$. It is clear that if $\mathrm{Id}:\s^2 \rightarrow \s^2$ is the identity map and $F:\s^2 \rightarrow \s^2$ is any anti-holomorphic isometry, then $(\mathrm{Id},F):\s^2 \times \s^2 \rightarrow \s^2 \times \s^2$ is a holomorphic isometry from $(\s^2 \times \s^2,\langle,\rangle,J_1)$ onto $(\s^2 \times \s^2,\langle,\rangle,J_2)$. Also, it is clear that any isometry of $\s^2\times\s^2$ is holomorphic or antiholomorphic with respect to $J_1$ or $J_2$.

The Kähler 2-form $\omega_j$ of $J_j$, $j=1,2$,
are given by $\omega_1=\pi_1^*\omega+\pi_2^*\omega$ and $\omega_2=\pi_1^*\omega-\pi_2^*\omega$ and hence
\[
\omega_1\wedge\omega_1=-\omega_2\wedge\omega_2=2(\pi_1^*\omega\wedge\pi_2^*\omega),
\]
where $\pi_j$, $j = 1, 2$, are the projections onto each factors. Along the paper, $\pi_1^*\omega\wedge\pi_2^*\omega$ will be the orientation on $\s^ 2\times\s^2$.

On the other hand, using that $\s^2 \times \s^2$ is a product manifold, the curvature tensor $\bar{R}$ of $\s^2 \times \s^2$ is given by
\begin{equation}\label{eq:curvature-tensor}
\begin{split}
\bar{R}(x, y, z, w) = &\frac{1}{2}\bigl( \prodesc{x}{w}\prodesc{y}{z} -\prodesc{x}{z}\prodesc{y}{w} + \\
&+ \prodesc{J_1 x}{J_2w}\prodesc{J_1y}{J_2 z} - \prodesc{J_1x}{J_2z}\prodesc{J_1y}{J_2 w}\bigr),
\end{split}
\end{equation}
and hence $\s^2\times\s^2$ is an Einstein manifold with scalar curvature $4$.

To unsderstand the construction of minimal surfaces of $\s^2\times\s^2$ made in Section~\ref{sec:gauss-map}, we need to identify $\s^2\times\s^2$ with the Grassmann manifold $G^+(2,4)$ of oriented $2$-planes in the Euclidean space $\R^4$. In fact, let $\Lambda^2\R^4=\{v\wedge w \,\lvert\, v,w\in\R^4\}\equiv\R^6$ be the space of $2$-vectors in $\R^4$ endowed with the Euclidean metric $\vprodesc{\,}{\,}$ given by
\[
\vprodesc{v\wedge w}{v'\wedge w'} = \langle v,v'\rangle\langle w,w'\rangle
-\langle v,w'\rangle\langle w,v'\rangle.
\]
Then the star operator $*:\Lambda^2\R^4\rightarrow\Lambda^2\R^4$ is an isometry and $\Lambda^2\R^4=\Lambda^2_+\R^4\oplus\Lambda^2_-\R^4$ where $\Lambda^2_{\pm}\R^4$ are the eigenspaces of $*$ with eigenvalues $\pm 1$. It is clear that, if $\{e_1,e_2,e_3,e_4\}$ is an oriented orthonormal frame of $\R^4$, then the frame $\{E^\pm_j:\, j = 1,2,3\}$ given by:
\begin{eqnarray*}
E^{\pm}_1=\frac{\textstyle 1}{\textstyle \sqrt 2}(e_1\wedge e_2\pm e_3\wedge e_4)=\frac{\textstyle 1}{\textstyle \sqrt 2}(e_1\wedge e_2\pm *(e_1\wedge e_2)),\\
E^{\pm}_2=\frac{\textstyle 1}{\textstyle \sqrt 2}(e_1\wedge e_3\pm e_4\wedge e_2)=\frac{\textstyle 1}{\textstyle \sqrt 2}(e_1\wedge e_3\pm *( e_1\wedge e_3)), \\
E^{\pm}_3=\frac{\textstyle 1}{\textstyle \sqrt 2}(e_1\wedge e_4\pm e_2\wedge e_3)=\frac{\textstyle 1}{\textstyle \sqrt 2}(e_1\wedge e_4\pm *( e_1\wedge e_4)),
\end{eqnarray*}
is a orthonormal oriented reference of $\Lambda^2_{\pm}\R^4$. We denote by $\s^2_{\pm}$ the unit spheres in the $3$-spaces $\Lambda^2_{\pm}\R^4$. We will denoted by $I:\Lambda^2_{+}\R^4\rightarrow\Lambda^2_{-}\R^4$ the isometry defined by $I(E^+_i)=E^-_i$. Also, if $A\in O(4)$ is an orthogonal matrix, then $\hat{A}:\Lambda^2\R^4\rightarrow\Lambda^2\R^4$ defined by $\hat{A}(v\wedge w)=Av\wedge Aw$ is an isometry satisfying $*\hat{A}=(\det A)\hat{A}*$. Hence if $\det A=1$, then $\hat{A}(\s^2_{\pm})=\s^2_{\pm}$ and if $\det A=-1$, then $\hat{A}(\s^2_{\pm})=\s^2_{\mp}$.

Finally, if $\{v,w\}$ is an oriented orthonormal frame of a plane $P\in G^+(2,4)$, then the map $G^+(2,4)\rightarrow \s^2_+\times\s^2_-$ given by
\[
P \mapsto \frac{\textstyle 1}{\textstyle \sqrt {2}}\left( v\wedge w+*(v\wedge w),v\wedge w-*(v\wedge w)\right),
\]
is a diffeomorphism.

Let $\Phi:\Sigma\rightarrow \s^2 \times \s^2$ be an immersion of an oriented surface $\Sigma$. Associated to the two Kähler structures of $\s^2 \times \s^2$ there exist two \emph{Kähler functions} $C_1,C_2:\Sigma\rightarrow\R$, defined by
\[
\Phi^*\omega_j=C_j\omega_{\Sigma},\quad j=1,2,
\]
where $\omega_{\Sigma}$ is the area $2$-form of $\Sigma$. The immersion $\Phi$ is called \textbf{complex} if it is complex with respect to $J_1$ or $J_2$, i.e., either $C_1^2 = 1$ or $C_2^2 = 1$. Also, the immersion $\Phi$ is called \textbf{Lagrangian} if it is Lagrangian with respect to $J_1$ or $J_2$, i.e., either $C_1 = 0$ or $C_2 = 0$. It is clear that $C_j^2\leq1$ and the points where $C_j^2=1$ are the complex points of $\Phi$ with respect to the complex structure $J_j$. It is interesting to remark that $C_j^2$ is well defined even when the surface is not orientable.

If $\Phi=(\Phi_1,\Phi_2)$, then it is easy to check that the Jacobians of $\Phi_1, \Phi_2: \Sigma \rightarrow \s^2$  are given by
\[
\hbox{Jac}\,(\Phi_1)=\frac{C_1+C_2}{2},\quad\hbox{Jac}\,(\Phi_2)=\frac{C_1-C_2}{2}.
\]
Moreover, if $\Sigma$ is compact, denoting by $d_j$ the degree of the map $\Phi_j:\Sigma\rightarrow\s^2$, i.e.,  $\int_{\Sigma}\hbox{Jac}\,\Phi_j\,dA=4\pi d_j$, we obtain that
\begin{equation}\label{eq:integral-formula-C}
\int_{\Sigma}C_1\,dA=4\pi(d_1+d_2),\quad \int_{\Sigma}C_2\,dA=4\pi(d_1-d_2).
\end{equation}

In particular, if $\Phi$ is a complex immersion, then its area $A = 4\pi m$, $m \in \N$, and if $\Phi$ is Lagrangian then $\abs{d_1} = \abs{d_2}$.

Using~\eqref{eq:curvature-tensor}, the Gauss equation of the immersion $\Phi$ is given by
\[
K = \frac{1}{2}(C_1^2 + C_ 2^2) + 2|H|^2 - \frac{|\sigma|^2}{2},
\]
where $K$ is the Gauss curvature of $\Sigma$, $\sigma$ is the second fundamental form of $\Phi$ and $H$ is the mean curvature vector. Using again~\eqref{eq:curvature-tensor}, the Codazzi equation is given by:
\[
(\nabla \sigma)(x, y, z) - (\nabla \sigma)(y,x,z) = \frac{1}{2}\bigl( \prodesc{J_1x}{J_2 z}(J_2J_1y)^\perp  - \prodesc{J_1y}{J_2z}(J_2J_1 x)^\perp\bigr),
\]
where $(\,)^\perp$ denotes the normal component to the immersion.

Finally, if $\{e_1,e_2,e_3,e_4\}$ is an oriented orthonormal local frame on $\Phi^*T(\s^2 \times \s^2)$ such that $\{e_1,e_2\}$ is an oriented frame on $T\Sigma$, the {\em normal curvature} $K^{\perp}$ of the immersion $\Phi$ is defined by
\[
K^{\perp}=R^{\perp}(e_1,e_2,e_4,e_3),
\]
where $R^{\perp}$ is the curvature tensor of the normal connection. Hence, using again~\eqref{eq:curvature-tensor}, we get the Ricci equation
\[
K^\perp = \frac{1}{2}(C_1^2 - C_2^2) + \prodesc{[A_{e_4}, A_{e_3}]e_1}{e_2},
\]
where $A_{\eta}$ stands for the Weingarten endomorphism associated to the normal vector $\eta$.

\emph{In what follows $\Phi = (\Phi_1, \Phi_2): \Sigma \rightarrow \s^2 \times \s^2$ will be a minimal immersion, that is, $H=0$.}

Now, we are going to describe the most regular examples of minimal surfaces of $\s^2\times\s^2$, that will be characterize along the paper.

\subsection{Complex surfaces}\label{subsec:complex-surfaces}
The simplest examples of minimal surfaces of $\s^2 \times \s^2$ are the complex ones, because it is well known that a complex submanifold of a Kähler manifold is always minimal. These surfaces are characterized by the fact that each component $\Phi_j$ of $\Phi$ is a conformal map (holomorphic or antiholomorphic) whose branch points are disjoint. More precisely, if $C_1 = 1$ (resp. $C_1 = -1$) then $\Phi_1$ and $\Phi_2$ are holomorphic (resp. antiholomorphic) maps and if $C_2 = 1$ (resp. $C_2 = -1)$ then $\Phi_1$ is holomorphic (resp.\ antiholomorphic) map and $\Phi_2$ is antiholomorphic (resp.\ holomorphic) map.

Moreover, because $(\s^2\times \s^2, J_1)$ and $(\s^2\times \s^2, J_2)$ are biholomorphic, complex surfaces with respect to $J_1$ and $J_2$ are congruent. Among complex surfaces we mention three important examples:

\begin{description}
\item[Slices] For any point $p\in\s^2$, the associate  slices are given by
\[
\begin{split}
\s^2\times\{p\}&=\{(x,p)\in\s^2\times\s^2\,|\, x\in\s^2\}, \\
\{ p\}\times\s^2&=\{(p,x)\in\s^2\times\s^2\,|\, x\in\s^2\}.
\end{split}
\]
These slices are totally geodesic surfaces with constant Gauss curvature $K = 1$, normal curvature $K^\perp = 0$ and  area $A = 4\pi$. $\s^2\times\{p\}$ has $C_1=C_2=1$ and $\{ p\}\times\s^2$ has  $C_1=-C_2=1$. Moreover, the slices are the only minimal surfaces of $\s^2 \times \s^2$ which are complex with respect to both complex structures $J_1$ and $J_2$ (see Proposition~\ref{prop:totallygeodesic}).

\item[Diagonal] Let $\textbf{D}=\{(x,x)\in\s^2\times\s^2\,|\,x\in\s^2\}$ be the diagonal of $\s^2\times\s^2$. It is clear that \textbf{\textbf{D}} is a totally geodesic surface with constant Gauss curvature $K = 1/2$, normal curvature $K^\perp = 1/2$ and area $A=8\pi$ (see \cite[Proposition 2.3]{CU}). Moreover $C_1=1$ and $C_2=0$ and so $\textbf{D}$ is a complex surface with respect to $J_1$ and Lagrangian with respect to $J_2$ (it is also characterized by this property, cf.\ Proposition~\ref{prop:totallygeodesic}).

\item[Weierstrass tori] Let $\Sigma = \C/\Lambda$ be the torus generated by the lattice $\Lambda = \{m + n\tau:\, m,n\in \Z\}$, with $\tau$ a complex number with $\mathrm{Im}\,\tau > 0$. Let $\wp:\Sigma \rightarrow \s^2$ be the Weierstrass $\wp$-function with a double pole at the origin. Then its branch points are exactly the $4$ half periods of $\Lambda$. Now, consider a point $p_0$ in $\Sigma$ which is not a branch point of $\wp$ and let $F$ be the automorphism of $\Sigma$ that maps $p_0$ to the origin. Then $\wp \circ F: \Sigma \rightarrow \s^2$ is the Weierstrass $\wp$-function with a double pole in $p_0$ and which branch points are disjoint with that of $\wp$. Therefore,
\begin{quote}
\itshape
$\Phi = (\wp, \wp \circ F): \Sigma \rightarrow (\s^2 \times \s^2, J_1)$ is a holomorphic embedding with area $16\pi$.
\end{quote}
In fact, since the degree of each component of $\Phi$ is $2$, equation~\eqref{eq:integral-formula-C} ensures that its area is $16\pi$. Hence, the only property that remain to prove is that $\Phi$ is an embedding. Let $p, q \in \Sigma$ two different points with $\Phi(p) = \Phi(q)$. Then, $\wp(p) = \wp(q)$ and $\wp(F(p)) = \wp(F(q))$ and so $p$ and $q$ are not branch points neither of $\wp$ nor $\wp\circ F$. Since $\wp$ is an even function it follows easily that $p_0$ is a branch point of $\wp$, which is a contradiction.
 \end{description}

\subsection{Lagrangian minimal surfaces}\label{subsec:lagrangian-surfaces}
Another important family of minimal surfaces of $\s^2 \times \s^2$ are the Lagrangian ones. As in the complex case, Lagrangian surfaces with respect to $J_1$ are congruent to Lagrangian surfaces with respect to $J_2$.

We have seen that the diagonal \textbf{D} is a Lagrangian surface with respect to $C_2$. Another important example is the following:

\begin{description}
\item[Clifford torus] Let $\textbf{T}=\{(x,y)\in\s^2\times\s^2\,|\,x_1=y_1=0\}$ be the product $\s^1\times\s^1$ embedded in $\s^2\times\s^2$ as the product of two equators. Then \textbf{T} is a totally geodesic surface, flat and with normal curvature $K^\perp = 0$. Moreover $C_1=C_2=0$ and \textbf{T} is the only minimal surface of $\s^2\times\s^2$ which is Lagrangian with respect to both complex structures $J_1$ and $J_2$ (cf.\ Proposition~\ref{prop:totallygeodesic}).
\end{description}

\section{Structure equations of minimal surfaces}
\label{sec:structure-equation}

In this section we are going to study the Frenet equations of a minimal immersion in $\s^2\times\s^2$ in order to obtain the compatibility equations. We will follow the arguments and notation developed in~\cite[Section 3]{TU1}.

Let $\Phi = (\Phi_1, \Phi_2): \Sigma \rightarrow \s^2 \times \s^2$ be a minimal immersion of an oriented surface $\Sigma$. We consider a local isothermal parameter $z=x+iy$ on $\Sigma$ such that
\begin{align*}
\langle\Phi_z,\Phi_z\rangle=\langle(\Phi_1)_z,(\Phi_1)_z\rangle+\langle (\Phi_2)_z, (\Phi_2)_z\rangle=0,\\
|\Phi_z|^2=|(\Phi_1)_z|^2+|(\Phi_2)_z|^2 = \frac{e^{2u}}{2},
\end{align*}
where the derivatives with respect to $z$ and $\bar{z}$ are given by $\partial_z=\frac{1}{2}\left(\frac{\partial}{\partial x}-i\frac{\partial}{\partial y}\right)$ and $\partial_{\bar z}=\frac{1}{2}\left(\frac{\partial}{\partial x}+i\frac{\partial}{\partial y}\right)$.

Let $\{N, \tilde{N}\}$ be an orthonormal reference in the normal bundle such that $\{\Phi_x, \Phi_y, \tilde{N}, N\}$ is an oriented reference in $\Phi^*T(\s^2 \times \s^2)$ and consider
\[
\xi=\frac{1}{\sqrt{2}}(N-i\tilde{N}).
\]
Then we have that $|\xi|^2=1, \langle\xi,\xi\rangle=0$ and  $\{\xi,\bar{\xi}\}$ is an orthonormal \emph{local reference} of the complexified normal bundle.  Following the same reasoning of~\cite[Section 3]{TU1} it can be proved that
\begin{eqnarray}
J_1\Phi_z=iC_1\Phi_z+\gamma_1\xi,\quad\quad J_1\xi=-2e^{-2u}\bar{\gamma}_1\Phi_z-iC_1\xi,\label{eq:J1Phi-xi}\\
J_2\Phi_z=iC_2\Phi_z+\gamma_2\bar{\xi},\quad\quad J_2\xi=-2e^{-2u}\gamma_2\Phi_{\bar{z}}+iC_2\xi, \label{eq:J2Phi-xi}
\end{eqnarray}
for certain local complex functions $\gamma_j,\,j=1,2$, which satisfy $|\gamma_j|^2=e^{2u}(1-C_j^2)/2$. Notice that if we choose another orthonormal oriented reference in the normal bundle then the functions $\gamma_j$ change but the property that $\prodesc{J_1 \Phi_z}{\xi} = \prodesc{J_2 \Phi_z}{\bar{\xi}} = 0$ remains true (see Remark~\ref{rmk:change-orthonormal-reference} for more details).

If $\hat{\Phi}: = (\Phi_1, -\Phi_2)$, then $\{\Phi, \hat{\Phi}\}$ is an orthogonal reference along $\Phi$ of the normal bundle of $\s^2 \times \s^2$ in $\R^6$. Also, from~\eqref{eq:J1Phi-xi} and \eqref{eq:J2Phi-xi}, it follows
\[
\hat{\Phi}_z = -J_1J_2\Phi_z = C_1C_2 \Phi_z + 2e^{-2u}\gamma_1\gamma_2\Phi_{\bar{z}} - iC_2\gamma_1 \xi - iC_1 \gamma_2 \bar{\xi}.
\]

Using the above information we easily get that the Frenet equations of the minimal immersion $\Phi$ are given by
\begin{align*}
\Phi_{zz}&=2u_z\Phi_z+f_1\xi+f_2\bar{\xi}-\frac{\gamma_1\gamma_2}{2}\hat{\Phi},\\
\Phi_{z\bar{z}}&=-\frac{e^{2u}}{4}\Phi - \frac{e^{2u}}{4}C_1C_2\hat{\Phi},\\
\xi_z&=-2e^{-2u}f_2\Phi_{\bar{z}}+ A \xi + \frac{iC_1\gamma_2}{2}\hat{\Phi},\\
\xi_{\bar{z}} &= - 2e^{-2u} \bar{f}_1 \Phi_{z} -\bar{A} \xi - \frac{iC_2\bar{\gamma}_1}{2}\hat{\Phi},
\end{align*}
for certain local complex functions $A$ and $f_j,\,j=1,2$.

We will call the \emph{fundamental data} of the pair $(\Phi,\xi)$ to the tuple
\[
(u, A, C_j, \gamma_j, f_j:\, j = 1, 2).
\]

\begin{remark}\label{rmk:change-orthonormal-reference}
Notice that if $\{\eta, \bar{\eta}\}$ is another orthonormal oriented reference in the complexify normal bundle then $\eta = e^{i\theta} \xi$ for certain funcion $\theta$. In such case the fundamental data $(u, A^*, C_j, \gamma_j^*, f_j^*:\, j = 1, 2)$ of the pair $(\Phi, \eta)$ are related with the initial ones by
\begin{equation}\label{eq:change-orthonormal-reference}
\gamma_1^* = e^{-i\theta} \gamma_1, \, \gamma_2^* = e^{i\theta}\gamma_2, \, f_1^* = e^{-i\theta} f_1,\, f_2^* = e^{i\theta}f_2 \text{ and } A^* = i\theta_z + A.
\end{equation}
\end{remark}

\begin{proposition}\label{prop:fundamental-equations}
Let $\Phi: \Sigma \rightarrow \s^2 \times \s^2$ be a minimal immersion of an orientable surface $\Sigma$ and $(u, A, C_j, \gamma_j, f_j:\, j = 1,2)$ its fundamental data for a given orthonormal reference. Then:
\begin{equation}
\left\{
\begin{aligned}
(C_j)_z&=2ie^{-2u}f_j\bar{\gamma}_j, &
(f_j)_{\bar{z}}&=i\frac{e^{2u}}{4}C_j\gamma_j + (-1)^{j+1}\bar{A}f_j,\\
(\gamma_j)_{\bar{z}}&=(-1)^{j+1}\bar{A}\gamma_j,&
|\gamma_j|^2&=\frac{e^{2u}(1-C_j^2)}{2},
\end{aligned}
\right. \, j=1,2.
\label{eq:compatibility}
\end{equation}

Moreover, if $\Phi$ is a complex immersion with respect to $J_1$ (resp.\ $J_2$) then $C_1^2 = 1$ and $\gamma_1 = f_1 = 0$ (resp.\ $C_2^2 = 1$ and $\gamma_2 = f_2 = 0$).
\end{proposition}

\begin{proof}
Firstly, the forth equation of~\eqref{eq:compatibility} comes from the definition of $\gamma_j$ (see equations~\eqref{eq:J1Phi-xi} and~\eqref{eq:J2Phi-xi}). Secondly, derivating with respect to $z$ and $\bar{z}$ in~\eqref{eq:J1Phi-xi} and~\eqref{eq:J2Phi-xi} and taking into account the Frenet equations we easily get the first and third equations of~\eqref{eq:compatibility}. Finally, from the $\xi$ and $\bar{\xi}$ components of $\Phi_{zz\bar{z}} = \Phi_{z\bar{z}z}$ we obtain the equation for $(f_j)_{\bar{z}}$, $j = 1, 2$.

Now, let suppose that $\Phi$ is a complex immersion with respect to $J_1$ (the case for $J_2$ is analogous). Then, the tangent and the normal bundle are invariant by $J_1$ and so $C_1^2 = 1$ and $\gamma_1 = 0$ from~\eqref{eq:J1Phi-xi}. It is well known that in this case $\sigma(J_1-, -) = J_1\sigma(-,-)$, where $\sigma$ is the second fundamental form of $\Phi$, and so $f_1 =\prodesc{\Phi_{zz}}{\bar{\xi}} =0$ since  $\xi = \frac{1}{\sqrt{2}}(N + iC_1J_1N)$.
\end{proof}

Conversely, we get the following result.

\begin{proposition}\label{prop:compatibility-equations}
Let $\Sigma$ be a simply connected surface, $u, C_j: \Sigma \rightarrow \R$ with $C_j^2 \leq 1$, $C_j^2$ not constant $1$, and $\gamma_j, f_j, A:\Sigma \rightarrow \C$, $j = 1, 2$, functions satisfying~\eqref{eq:compatibility}. Then there exists, up to congruences, a unique non-complex minimal immersion $\Phi: \Sigma \rightarrow \s^2 \times \s^2$  and an orthonormal reference of the complexified normal bundle $\{\xi,\bar{\xi}\}$ whose fundamental data are $(u, A, C_j, \gamma_j, f_j:\, j = 1, 2)$.
\end{proposition}

\begin{proof}
First, the equation for $(\gamma_j)_{\bar{z}}$ in~\eqref{eq:compatibility} ensures that $\gamma_j$ can be written as the product of a positive function and a holomorphic one and so $\gamma_j$ has isolated zeroes since we have suppose that $C_j^2 \not\equiv 1$, $j = 1, 2$. Therefore, the set $\hat{\Sigma} = \{p\in \Sigma:\, \gamma_1 \neq 0, \gamma_2 \neq 0\}$ is dense in $\Sigma$.

Now, deriving with respect to $z$ in the last equation of~\eqref{eq:compatibility} we get
\[
\bar{\gamma}_j\bigl[(\gamma_j)_z - (2u_z + (-1)^{j}A) \gamma_j + 2iC_jf_j\bigr] = 0, \quad j = 1, 2.
\]
As we have noticed, we can simplify $\bar{\gamma}_j$ and obtain an equation for $(\gamma_j)_z$ that holds in $\Sigma$. Using this new equation joint with~\eqref{eq:compatibility} in the equality $(\gamma_j)_{\bar{z}z} = (\gamma_j)_{z\bar{z}}$ we get that
\begin{equation}\label{eq:Gauss-Ricci}
\gamma_j\left(2u_{z\bar{z}} - 4e^{-2u}|f_j|^2 + \frac{e^{2u}}{2}C_j^2 + (-1)^{j}(\bar{A}_z + A_{\bar{z}})\right) = 0, \quad j=1,2.
\end{equation}
Finally, from the equations~\eqref{eq:Gauss-Ricci}, once we simplified the term $\gamma_j$, and~\eqref{eq:compatibility} we can easily check that
$\Phi_{zz\bar{z}} = \Phi_{z\bar{z}z}$ and $\xi_{z\bar{z}} = \xi_{\bar{z}z}$, which are the integrability conditions of the Frenet system.
\end{proof}

\begin{definition}\label{def:Hopf-differential}
Let $\Phi = (\Phi_1, \Phi_2): \Sigma \rightarrow \s^2 \times \s^2$ be a minimal immersion of an oriented surface $\Sigma$. We define the Hopf $2$-differential $\Theta$ associated to $\Phi$ by
\begin{equation}\label{eq:modulo-Hopf-differential}
\Theta(z) = \frac{1}{2}\prodesc{J_1\Phi_z}{J_2\Phi_z} \df z \otimes \df z
\end{equation}
where $z$ is a conformal parameter in $\Sigma$.
\end{definition}

Observe that $\Theta = \Theta_{1} = -\Theta_{2}$, where $\Theta_{j}=\prodesc{(\Phi_j)_z}{(\Phi_j)_z} \df z \otimes \df z$ is the Hopf $2$-differential associated to the harmonic map $\Phi_j$, $j = 1, 2$, and so it is holomorphic. Locally,  taking into
account the fundamental data of the immersion, we can write $\prodesc{J_1\Phi_z}{J_2\Phi_z} = \gamma_1 \gamma_2$ and so
\begin{equation}\label{eq:modulo-diferencial-Hopf}
|\langle J_1\Phi_z,J_2\Phi_z\rangle|^2 = |\gamma_1|^2|\gamma_2|^2=\frac{e^{4u}}{4} (1-C_1^2)(1-C_2^2).
\end{equation}

Since $(\gamma_j)_{\bar{z}} = (-1)^{j+1}\bar{A}\gamma_j$, either $\gamma_j$ is identically zero or its zeroes are isolated (see the proof of Proposition~\ref{prop:compatibility-equations}). Hence, either $C_j^2 = 1$ in $\Sigma$ or the points where $C_j^2 = 1$ are isolated, $j =1, 2$.

In the following lemma we are going to get some equations that will be usefull in the sequel.

\begin{lemma}\label{lm:properties-fundamental-data}
Let $\Phi: \Sigma \rightarrow \s^2 \times \s^2$ be a minimal immersion of an orientable surface $\Sigma$ with  fundamental data $(u, A, C_j, \gamma_j, f_j:\, j = 1, 2)$ for a given orthonormal reference. Then:
\begin{enumerate}[(i)]
	\item $\abs{f_j}^2 = \frac{e^{4u}}{8}(C_j^2 - K + (-1)^{j}K^\perp)$. \label{eq:modulo-f}
	In particular, if $\Phi$ is a complex immersion with respect to $J_j$ then $K + (-1)^{j+1}K^\perp = 1$.
	\item The Laplacian and the gradient of $C_j$ are given by \label{eq:laplacian-gradient-C}
	\[
	\begin{split}
		\Delta C_j &= 2C_j(K + (-1)^{j+1}K^\perp) - C_j(1 + C_j^2),  \\
		\abs{\nabla C_j}^2 &= (1-C_j^2)(C_j^2 - K + (-1)^{j}K^\perp).
	\end{split}
	\]
	\item \label{eq:laplacian-logarithm}
	$\Delta \log(1\pm C_j) = \mp C_j + K + (-1)^{j+1}K^\perp$, whenever the function $1 \pm C_j \not\equiv 0$.

	\item \label{item:formulas-caracteristica-fibrado-normal} If $\Sigma$ is compact and $\Phi$ is non-complex then:
	\[
	\begin{aligned}
	-N_1^+=2(d_1+d_2)+\chi+\chi^{\perp},\quad-N_1^-=-2(d_1+d_2)+\chi+\chi^{\perp}\\
	-N_2^+=2(d_1-d_2)+\chi-\chi^{\perp},\quad-N_2^-=-2(d_1-d_2)+\chi-\chi^{\perp}
	\end{aligned}
	\]
	where $N_i^{\pm}$ are the sum of  the orders for all the zeroes of the functions $1\mp C_i$ and $\chi$, $\chi^{\perp}$ are the Euler numbers of $\Sigma$ and the normal bundle of $\Phi$. When $\Phi$ is a complex immersion with respect to $J_i$ but not with respect to $J_j$, $i \neq j$, only the equations for $N_j^\pm$ hold.
\end{enumerate}
\end{lemma}

\begin{proof}
If $\Phi$ is a complex immersion with respect to $J_j$ then we already know that $C_j^2 = 1$ and $f_j = 0$ (cf.\ Proposition~\ref{prop:fundamental-equations}). From Gauss and Ricci equations it is easy to deduce that $K + (-1)^{j+1}K^\perp = 1$ and (i) follows in this case. If $\Phi$ is a non-complex immersion then, using that $K^\perp = 2e^{-2u}(\bar{A}_z + A_{\bar{z}})$, \eqref{eq:Gauss-Ricci} and $K = -4e^{-2u}u_{z\bar{z}}$, we obtain (i) in this case.

Taking into account~\eqref{eq:compatibility}, \eqref{eq:modulo-f} and the expressions $\Delta f = 4e^{-2u}f_{z\bar{z}}$ and $|\nabla f|^2 = 4e^{-2u}|f_z|^2$, we deduce (ii). Now, (iii) follows easily from (ii). Finally, (iv) is proved integrating~\eqref{eq:laplacian-logarithm} and using~\cite[formulae (6) and (7) in \S 2]{SY}.
\end{proof}

\section{Local results}\label{sec:local-results}
We start this section characterizing locally the easiest examples of minimal surfaces of $\s^2\times\s^2$: the slices, the diagonal and the Clifford torus (see Section~\ref{subsec:complex-surfaces} and~\ref{subsec:lagrangian-surfaces} for their definition).
Although their classification as the only totally geodesic surfaces is well-known (see \cite{CN}) we include it for completeness.

\begin{proposition}\label{prop:totallygeodesic}
Let $\Phi = (\Phi_1, \Phi_2):\Sigma\rightarrow\s^2\times\s^2$ be an immersion of an orientable surface $\Sigma$. Then
the following statements are equivalent:
\begin{enumerate}
\item $\Phi$ is totally geodesic,
\item $\Phi$ is minimal and the Gauss and normal curvatures $K$ and $K^{\perp}$ are constant,
\item $\Phi$ is minimal and the functions $C_1$ and $C_2$ are constant,
\item $\Phi(\Sigma)$ is congruent to an open subset of either an slice, or
\textbf{D} or \textbf{T}.
\end{enumerate}
\end{proposition}

\begin{proof}
As we remarked in Sections~\ref{subsec:complex-surfaces} and~\ref{subsec:lagrangian-surfaces}, the slices and the surfaces \textbf{D} and \textbf{T}  are totally geodesic and the corresponding functions $C_1$, $C_2$, $K$ and $K^{\perp}$ are constant.

If $\Phi$ is a totally geodesic immersion, i.e.\ $f_1 = f_2 = 0$, then $\Phi$ is a minimal immersion and from Lemma~\ref{lm:properties-fundamental-data}.\eqref{eq:modulo-f}-\eqref{eq:laplacian-gradient-C}
 we deduce that $K$ and $K^{\perp}$ are constant.

If $\Phi$ is a minimal immersion and $K$ and $K^{\perp}$ are constant functions, we define  $a_j=K+(-1)^{j+1}K^{\perp},\, j=1,2$. Then, from Lemma~\ref{lm:properties-fundamental-data}.\eqref{eq:laplacian-gradient-C} we obtain that
\[
|\nabla C_j|^2=(1-C_j^2)(C_j^2-a_j),\quad \Delta C_j=C_j(2a_j-1-C_j^2),\quad j=1,2.
\]
These equations mean that $C_j$ are isoparametric functions on $\Sigma$. Let $U_j = \{p\in \Sigma\,|\, \nabla (C_j)_p \neq 0\}$, $j = 1, 2$.
If $U_j=\emptyset$, then $C_j$ is constant and the above equations imply that either $a_j=0$ or $a_j=1$, for any $j\in\{1,2\}$. If $U_j\not =\emptyset$, then using~\cite[Lemma 3.3]{EGT}, we obtain that
\[
(-9a_j+a_k+8)C_j^2 = a_j(3a_j+a_k-4),\quad \text{in } U_j,
\]
for any $j\in\{1,2\}$, where $k\in\{1,2\},\,k\not= j$.
Since $C_j^2$ is non-constant in $U_j$, from the last equation we get $9a_j=a_k+8$ and hence either $a_j = 0$ or $3a_j+a_k=4$. As $a_j\not= 1$, we have that $a_j=0, a_k=-8$.

Hence the only positility is that $U_1=U_2=\emptyset$. Therefore, both $C_1$ and $C_2$ are constant functions.

 Finally, if $\Phi$ is a minimal immersion and $C_1$ and $C_2$ are constant functions,  then from Lemma~\ref{lm:properties-fundamental-data}.\eqref{eq:laplacian-gradient-C} it is easy to check that the only value for the constant functions $C_1$ and $C_2$ are $1$, $-1$ or $0$. If $C_1=C_2=0$, then $\Phi(\Sigma)$ is, up to a congruence, an open subset of \textbf{T} (see \cite[Proposition 2.2]{CU}). If $C_i^2=1$, $C_j=0$, $i\not=j$, then $\Phi(\Sigma)$ is, up to a congruence, an open subset of \textbf{D} (see \cite[Proposition 2.3]{CU}). Finally if $C_1^2=C_2^2 = 1$, then either $\hbox{rank}\,(\Phi_1)=0$ at any point or $\hbox{rank}\,(\Phi_2)=0$ at any point. So, either $\Phi_1$ or $\Phi_2$ is a constant function and so $\Phi(\Sigma)$ is an open piece of a slice.
\end{proof}

A totally geodesic hypersurface of $\s^2 \times \s^2$ is, up to an ambient isometry, an open set of $\s^2 \times \s^1$ (cf.~\cite[Proposition 1]{TU1}). Hence, any minimal surface of $\s^2 \times \s^1$ or of its universal cover $\s^2 \times \R$ is also a minimal surface of $\s^2 \times \s^2$. In the following result we characterize those minimal surfaces of $\s^2 \times \s^2$ that lay in a totally geodesic hypersurface.

\begin{proposition}\label{prop:curvaturanormalnula}
Let $\Phi:\Sigma\rightarrow\s^2\times\s^2$ be a minimal immersion of a surface $\Sigma$. Then, $\Phi(\Sigma)$ is non-full, i.e.\ it is contained in a totally geodesic hypersurface of $\s^2 \times \s^2$, if and only if $C_1^2 = C_2^2$. Moreover, these surfaces have $K^\perp = 0$.
\end{proposition}

\begin{proof}
We will assume that $\Sigma$ is oriented taking the oriented two-fold covering of $\Sigma$ if necessary.

Let suppose that $\Phi$ factorizes through $\s^2 \times \s^1$. Hence there exists $a \in \R^3$, $|a| = 1$, such that $\tilde{N} = (0, a)$ is a unit normal vector to $\s^2 \times \s^1$ in $\s^2 \times \s^2$. Let $N$ be a unit normal vector to $\Phi$ in $\s^2 \times \s^1$ such that $\{\tilde{N}, N\}$ is an oriented orthonormal reference in the normal bundle with the induced orientation. Following the previous section we consider the reference in the complexify normal bundle $\sqrt{2}\xi = N - i\tilde{N}$. Therefore:
\[
0 = i\sqrt{2} \tilde{N}_z =  (\bar{\xi} - \xi)_z = 2e^{-2u}(f_2 - f_1)\Phi_{\bar{z}} - A(\xi + \bar{\xi}) +\frac{i}{2}(C_2\gamma_1 - C_1\gamma_2)\hat{\Phi}
\]
where we have used the Frenet equations of $\Phi$. From the last equation we deduce that $A = 0$ and $f_1 = f_2$. Hence $K^\perp = 2e^{-2u}(\bar{A}_z + A_{\bar{z}}) = 0$ (cf.\ proof of Lemma~\ref{lm:properties-fundamental-data}) and, using Lemma~\ref{lm:properties-fundamental-data}.\eqref{eq:modulo-f}, we get $C_1^2 = C_2^2$.

Conversely, let suppose that $C_1^2 = C_2^2$. If $\Phi$ is a complex immersion, it is clear that $\Phi(\Sigma)$ is an slice of $\s^2 \times \s^2$ and so, factorize through a totally geodesic hypersurface. Hence, we are going to suppose that $\Phi$ is not a complex immersion.

Deriving in the equation $C_1^2 = C_2^2$ with respecto to $z$ and taking into account~\eqref{eq:compatibility} we get $C_1 f_1 \bar{\gamma}_1 = C_2 f_2 \bar{\gamma}_2$. Using now Lemma~\ref{lm:properties-fundamental-data}.\eqref{eq:modulo-f} we deduce $C_1^2K^\perp = 0$. Let $U = \{p \in \Sigma\, |\, C_1 = 0\}$. If the interior of $U$ is non empty, then $\mathrm{int}(U)$ is a open piece of $\mathbf{T}$ (cf.\ Proposition~\ref{prop:totallygeodesic} and so $K^\perp = 0$ on $\mathrm{int}(U)$. Hence $K^\perp = 0$ on $\Sigma$.

Now, using Lemma~\ref{lm:properties-fundamental-data}.\eqref{eq:laplacian-gradient-C} we deduce that
\[
(\Delta + F)(C_1 - C_2) = 0, \qquad F = -2K + (1+C_1^2) = -2K + (1+C_2^2)
\]
Using classical results from elliptic theory (cf.\ ~\cite{C}), we obtain that either $C_1 = C_2$ or $A = \{p \in \Sigma\,|\, C_1 (p) = C_2 (p)\}$ is a set of curves in $\Sigma$. In this second case, since $C_1^2 =
C_2^2$, we have that $C_1 + C_2 = 0$ on $\Sigma \setminus A$ and hence on $\Sigma$. To sum up we have two possibilities: $C_1 = C_2$ or $C_1 = -C_2$. It is clear that the surfaces with $C_1 = -C_2$ can
be obtained as the images of the surfaces with $C_1 = C_2$ under the isometry $S$ of $\s^2 \times \s^2$ given by $S(p, q) = (q, p)$. Hence, we can assume that, up to an ambient isometry, $C_1 = C_2$.

Let $(U, z = x+iy)$ a simply connected complex neighbourhood of $\Sigma$. We claim that there exist a normal reference on $U$ such that $A = 0$ (see Section~\ref{sec:structure-equation}). In fact, consider the real $1$-form $\alpha = -\mathrm{Im}(Adz)$. Since $0 = e^{2u}K^\perp = 2(\bar{A}_z + A_{\bar{z}})$, the $1$-form $\alpha$ is closed and so $\alpha = d\theta$. Hence, $A = -2i\theta_z$. Now, in the reference $\{e^{2i\theta}\xi, e^{-2i\theta}\bar{\xi}\}$ the function $A$ vanishes (cf.\ Remark~\ref{rmk:change-orthonormal-reference}). Therefore we can suppose, up to a change in the orthonormal reference, that $A = 0$.

Since $A = 0$, $\gamma_1, \gamma_2: U \rightarrow \C$ are holomorphic functions  with $|\gamma_1|^2 = |\gamma_2|^2$ (cf.~equations~\eqref{eq:compatibility}). So, there exists $\varphi \in \R$ such that $\gamma_2 = e^{i\varphi}\gamma_1$. Now, in the normal reference $\{e^{-i\varphi/2}\xi, e^{i\varphi/2}\bar{\xi}\}$ the new funcions $\gamma_1$ and $\gamma_2$ are equal and, as $\varphi$ is constant, the property $A = 0$ is still satisfied (cf.\ Remark~\ref{rmk:change-orthonormal-reference}).

Moreover, deriving with respecto to $z$ in the equation $C_1 = C_2$ and taking into account~\eqref{eq:compatibility} we get $ f_1 \bar{\gamma}_1 = f_2 \bar{\gamma}_2$ and so $f_2 = f_1$. Now, from the Frenet equations of $\Phi$, it follows that $(\xi - \bar{\xi})_z = 0$ and hence
\[
\tilde{N}=\frac{i}{\sqrt{2}}(\xi-\bar\xi):\Sigma\rightarrow\R^6
\]
is a constant map. Finally, from the equations~\eqref{eq:J1Phi-xi} and \eqref{eq:J2Phi-xi} we obtain that
$J_1\tilde{N}+J_2\tilde{N}=0$ and so, $\tilde{N}=(0,a)\in \s^2\times\s^2$. So $\Phi(U)$ is contained in $\s^2\times\s^1_a$, where $\mathbb{S}^1_a$ is the great circle orthogonal to the vector $a$.
To finish the proof, we only need a connecteness argument.
\end{proof}

In the rest of the section the relation between minimal surfaces of $\s^2 \times \s^2$ and the sinh-Gordon equation is shown.

\begin{proposition}\label{prop:sinh-Gordon}
Let $\Phi:\Sigma\rightarrow \s^2 \times \s^2$ be a minimal immersion of an orientable surface without complex points. Then, there exist smooth functions $v,w:\Sigma \rightarrow \R$ satisfying:
\[
v_{z\bar{z}} + \frac{|\langle J_1\Phi_z,J_2\Phi_z\rangle|}{4}\sinh(2v) = 0 \quad \text{and} \quad w_{z\bar{z}} +\frac{|\langle J_1\Phi_z,J_2\Phi_z\rangle|}{4}\sinh(2w) = 0,
\]
where $z = x+iy$ is any conformal parameter on $\Sigma$.
\end{proposition}

\begin{proof}
Since $\Phi$ has not complex points, the functions $C_j$ satisfy $C_j^2< 1$ and so we can define
$v, w:\Sigma \rightarrow \R$ by
\[
2v = \log \sqrt{\frac{(1+C_1)(1+C_2)}{(1-C_1)(1-C_2)}}, \qquad 2w = \log \sqrt{\frac{(1-C_1)(1+C_2)}{(1+C_1)(1-C_2)}},
\]
i.e., $C_1 = \tanh(v - w)$ and $C_2 = \tanh(v+w)$. Now, using Lemma~\ref{lm:properties-fundamental-data}.\eqref{eq:laplacian-logarithm}, we get that
\[
\begin{split}
\Delta v &= -\frac{1}{2}\bigl( \tanh(v+w) + \tanh(v-w)\bigr), \\
\Delta w &= -\frac{1}{2}\bigl(\tanh(v+w) - \tanh(v-w)\bigr).
\end{split}
\]

Let $z = x+iy$ be a conformal parameter on $\Sigma$. Then, from~\eqref{eq:modulo-diferencial-Hopf}, we get
\[
e^{2u} = \frac{2|\langle J_1\Phi_z,J_2\Phi_z\rangle|}{\sqrt{(1-C_1^2)(1-C_2^2)}} = 2|\langle J_1\Phi_z,J_2\Phi_z\rangle|\cosh(v+w)\cosh(v-w).
\]
Finally, taking into account that in local coordinates $\Delta f = 4e^{-2u}f_{z\bar{z}}$, we get the equations for $v_{z\bar{z}}$ and $w_{z\bar{z}}$.
\end{proof}

Conversely, we obtain the following result.

\begin{theorem}\label{thm:sinh-Gordon->minima-s2xs2}
Let $v,w:\C\rightarrow\R$ be two solutions of the sinh-Gordon equation
 \[
v_{z\bar{z}} + \frac{1}{2}\sinh(2v) = 0 \quad \text{and} \quad w_{z\bar{z}} +\frac{1}{2}\sinh(2w) = 0.
\]
Then there exists a $1$-parameter family of minimal immersions $\Phi_{t}:\C\rightarrow\s^2\times\s^2$ without complex points, whose induced metric is $4\cosh (v+w)\cosh (v-w)|dz|^2$, whose Kähler functions are $C_1 = \tanh(v-w)$ and $C_2 = \tanh(v+w)$ and whose holomorphic Hopf $2$-differentials are $\Theta_{t}=e^{it}dz\otimes dz$.
\end{theorem}

\begin{remark}~\label{rmk:sinh-Gordon->minima-s2xs2}
\begin{enumerate}
	\item The immersions $\Phi_{t}$ are Lagrangian if and only if $w=\pm v$.
	\item \label{rmk:item:sinh-Gordon->minima-s2xR} From Proposition~\ref{prop:curvaturanormalnula}, the minimal immersions $\Phi_t$ are non-full in $\s^2\times\s^2$ if and only if $v = 0$ or $w = 0$. In this particular case, associated to any solution $v$ of the sinh-Gordon equation, there exists a $1$-parameter family of minimal immersions $\Phi_t$ in $\s^2\times \R$ with induced metric $4\cosh^2 v |dz|^2$. This is a well known result~\cite{HKS}.
\end{enumerate}
\end{remark}

\begin{proof}
For $j=1,2$, we define functions $u,C_j:\C\rightarrow\R$ by
\[
C_1 = \tanh(v - w),\, C_2 = \tanh(v+w) \text{ and } e^{2u} = 4\cosh(v+w)\cosh(v-w).
\]

Moreover, we also defined $\gamma_j, A:\C\rightarrow \C$ by
\begin{gather*}
\gamma_1 =\sqrt 2e^{it/2}\sqrt\frac{\cosh(v+w)}{\cosh(v-w)},\quad
\gamma_2=\sqrt 2e^{it/2}\sqrt\frac{\cosh(v-w)}{\cosh(v+w)}, \\
A=\left(\log\sqrt{\frac{\cosh(v+w)}{\cosh(v-w)}}\right)_{z}.
\end{gather*}

Then, it is easy to check that $2|\gamma_j|^2 = e^{2u}(1-C_j^2)$ and $(\gamma_j)_{\bar{z}} = (-1)^{j+1}\bar{A}\gamma_j$, $j = 1, 2$.

Also, we defined the functions $f_j:\C \rightarrow\C$ by  $f_j = -i\gamma_j(v+(-1)^jw)_z$, $j = 1, 2$,
and so these functions satisfy $(C_j)_z = 2ie^{-2u}f_j\bar{\gamma_j}$. Taking into account that $v$ and $w$ are solutions to the sinh-Gordon equation, we also get that $(f_j)_{\bar{z}} = i\frac{e^{2u}}{4}C_j\gamma_j + (-1)^{j+1}\bar{A}f_j$, $j = 1, 2$.

Therefore, we have proved that, for each $t\in \R$, the tupla $(u, A, C_j, \gamma_j, f_j:\, j = 1, 2)$ fulfills the compatibility equations~\eqref{eq:compatibility} and so, from Proposition~\ref{prop:compatibility-equations}, is the fundamental data of a unique minimal immersion $\Phi_t:\C \rightarrow \s^2 \times \s^2$. Moreover, the immersion $\Phi_t$ has no complex points because, by definition, $C_1^2, C_2^2 < 1$,  the induced metric by $\Phi_t$ is $e^{2u}|dz|^2 = 4\cosh(v+w)\cosh(v-w)|dz|^2$ and the Hopf 2-differential of $\Phi_t$ is given by $\Theta =\frac{1}{2} \gamma_1\gamma_2 dz\otimes dz = e^{it}dz\otimes dz$.
\end{proof}

\section{Gauss map of pair of minimal surfaces of $\s^3$}
\label{sec:gauss-map}
Let $\Sigma$ be a Riemann surface and $\phi:\Sigma\rightarrow\s^3$ a conformal minimal immersion of $\Sigma$ in the $3$-dimensional unit sphere. Following \cite{H,L}, the Hopf differential associated to $\phi$ is the holomorphic $2$-differential on $\Sigma$ defined by
\[
\Theta_\phi(z)=\langle \phi_z,N_z\rangle dz\otimes dz,
\]
where $z = x + iy$ is a complex parameter on $\Sigma$ and $N$ is the unit normal vector field to $\phi$ such that $\{\phi_x, \phi_y, \phi, N\}$ is a positively oriented frame in $\R^4$.

Considering $\s^3\subset\R^4$, the Gauss map of $\phi:\Sigma\rightarrow\R^4$ is the map $$(\nu^+_\phi,\nu^-_\phi):\Sigma\rightarrow\s^2_+\times\s^2_-$$ defined by
\[
\nu^{\pm}_\phi(p)=\frac{1}{\sqrt 2}(e_1\wedge e_2 \pm\phi(p)\wedge N_p),
\]
where $\{e_1,e_2\}$ is an oriented orthonormal basis in $T_p\Sigma$ (see Section~\ref{sec:preliminares}). Since $\phi: \Sigma \rightarrow \R^4$ has parallel mean curvature vector, from~\cite{RV}, $\nu^\pm_\phi$ are harmonic maps from $\Sigma$ into $\s^2$.

Also, it is interesting to remark a well-known result (see~\cite{H}) that will be used later.
\begin{quote}
{\it For any solution $v:\C \rightarrow \R$ of the sinh-Gordon equation $v_{z\bar{z}} + \frac{1}{2}\sinh(2v) = 0$ there exists a $1$-parameter family $\phi_t:\C \rightarrow \s^3$ of minimal immersions whose induced metric is $e^{2v}\abs{dz}^2$ and whose Hopf differential is $\Theta_{\phi_t}(z) = \frac{i}{2}e^{it}dz\otimes dz$}.
\end{quote}
\begin{theorem}\label{thm:aplicacion-gauss-minima-s3}
Let $\Sigma$ be a Riemann surface and $\phi,\psi:\Sigma\rightarrow\s^3$ two conformal minimal immersions with the same Hopf differentials $\Theta_{\phi}=\Theta_{\psi}$ and induced metrics $g_{\phi}$ and $g_{\psi}$ respectively. Then
\[
\nu_{\{\phi, \psi\}}=(\nu^+_{\phi},\nu^-_{\psi}):\Sigma\rightarrow\s^2_+\times\s^2_-
\]
 is a conformal minimal immersion. Moreover, the induced metric by $\nu_{\{\phi, \psi\}}$ is
\[
g=\frac{1}{2}\Bigl((2+|\sigma_{\phi}|^2)g_{\phi}+ (2+|\sigma_{\psi}|^2)g_{\psi}\Bigr),
\]
the Kähler functions are
 \[
 C_1=\frac{|\sigma_{\psi}|-|\sigma_{\phi}|}{|\sigma_{\psi}|+|\sigma_{\phi}|},\quad C_2=\frac{2-|\sigma_{\phi}||\sigma_{\psi}|}{2+|\sigma_{\phi}||\sigma_{\psi}|},
 \]
and the Hopf differential is $\Theta = -2i\Theta_\phi = -2i\Theta_\psi$, where $|\sigma_{\phi}|$  and $|\sigma_{\psi}|$ are computed with the metric $g_{\phi}$ and $g_{\psi}$ respectively.  We will say that $\nu_{\{\phi, \psi\}}$ is the Gauss map of the pair of minimal immersions $\{\phi,\psi\}$.
\end{theorem}
\begin{remark}
As $\Theta_{\phi}=\Theta_{\psi}$, the zeroes of $\sigma_{\phi}$ and $\sigma_{\psi}$ coincide and hence the functions $C_1$ and $C_2$ satisfy $-1 < C_1 < 1$ and $-1 < C_2 \leq 1$. Moreover, $C_2(p) = 1$ if and only if $\sigma_\phi(p) = \sigma_\psi(p) = 0$.
\end{remark}

\begin{remark}\label{rmk:aplicacion-gauss-minima-s3}~
\begin{enumerate}[(i)]
	\item If $\psi = \phi$, then $\nu_{\{\phi, \phi\}}$ is the Gauss map of $\phi$ and $C_1 = 0$, i.e., $\nu_{\{\phi, \phi\}}$ is a Lagrangian immersion. This case was studied in~\cite{CU} where it was showed that the diagonal $\mathbf{D}$ is the Gauss map of the totally geodesic surface $\s^2  \subset \s^3$ and also that the Gauss map of the Clifford torus $\{(z, w) \in \s^3 \subset \C^2\, |\, |z| = |w| = 1/\sqrt{2} \}$ is a two-fold covering of the totally geodesic torus $\mathbf{T}$.
	\item If $S$ is the isometry of $\s^2\times\s^2$ given by $S(p,q)=(q,p)$, then $\nu_{\{\psi, \phi\}}=S\circ (I,I^{-1})\circ\nu_{\{\phi, \psi\}}$, where $I:\s^2_+\rightarrow\s^2_-$ is the isometry defined in Section~\ref{sec:preliminares}.
	\item Given $A \in O(4)$, then $\nu_{\{A\phi, \psi\}}$ is congruent to $\nu_{\{\phi, \psi\}}$. More precisely, if $\det A = 1$ then $\nu_{\{A\phi, \psi\}} = (\hat{A}\times \mathrm{Id})\circ \nu_{\{\phi, \psi\}}$ and if $\det A = -1$ then $\nu_{\{A\phi, \psi\}} = \bigl((\hat{A}\circ I) \times \mathrm{Id}\bigr)\circ \nu_{\{\phi, \psi\}}$, where $\hat{A}: \Lambda^2 \R^4 \rightarrow \Lambda^2 \R^4$ was defined in Section~\ref{sec:preliminares} and $\mathrm{Id}$ denotes the identity map.
\end{enumerate}
\end{remark}

\begin{remark}\label{rmk:aplicacion-gauss-ramificacion}
Following the proof of Theorem~\ref{thm:aplicacion-gauss-minima-s3}, if only one of the immersions $\phi$ or $\psi$ is branched, then $\nu_{\{\phi, \psi\}}$ is also a conformal minimal immersion in $\s^2 \times \s^2$. Such situation happens when one consider a minimal immersion $\phi: \Sigma \rightarrow \s^3$ and $\psi$ is its \emph{polar} immersion (possibly branched) $N: \Sigma \rightarrow \s^3$, where $N$ is a unit normal vector field to $\phi$ (cf.~\cite[Proposition 10.1]{L}). $N$ is also a minimal immersion and since both of them have the same Hopf differentials, then the Gauss map $\nu_{\{\phi, N\}}$ is a minimal immersion of $\Sigma$ in $\s^2 \times \s^2$. Nevertheless, it is easy to check that $\nu^-_N = -\nu^-_\phi$ and so, the Gauss map of the pair $\{\phi, N\}$ is congruent, by the isometry $\mathrm{Id}\times -\mathrm{Id}$ of $\s^2 \times \s^2$, to the Gauss map of $\phi$.
\end{remark}

\begin{proof}
Let $z = x + iy$ a complex coordinate in $\Sigma$ and $N$ and $\hat{N}$ be the unit normal vector fields to $\phi$ and $\psi$ respectively such that $\{\phi_x,\phi_y,\phi, N\}$ and $\{\psi_x,\psi_y,\psi, \hat{N}\}$ are positively oriented frames in $\R^4$. The induced metric by $\phi$ and $\psi$ are given by
$g_\phi = e^{2v}|dz|^2$ and $g_\psi = e^{2w}|dz|^2$. Also, we can write the Hopf differentials $\Theta_\phi(z) = \Theta_\psi(z) = \theta(z) dz \otimes d z$.

The Frenet equations of $\phi$ and $\psi$ are given by (see~\cite{H,L}):
\begin{align*}
\phi_{zz} &= 2v_z\phi_z-\theta N, & \phi_{z\bar{z}} &= -e^{2v}\phi/2,& N_z &= 2e^{-2v}\theta\phi_{\bar{z} }, \\
\psi_{zz}&=2w_z\psi_z-\theta\hat{N},& \psi_{z\bar{z}}&=-e^{2w}\psi/2, & \hat{N}_z&=2e^{-2w}\theta\psi_{\bar{z}}.
\end{align*}
On the other hand, the components of the Gauss map $\nu^+_{\phi}$ and $\nu^-_{\psi}$ can be written as
\begin{eqnarray*}
\nu^+_{\phi}(z)=\frac{1}{\sqrt 2}(-2ie^{-2v}\phi_z\wedge \phi_{\bar{z}}+\phi\wedge N),\\
\nu^-_{\psi}(z)=\frac{1}{\sqrt 2}(-2ie^{-2w}\psi_z\wedge \psi_{\bar{z}}-\psi\wedge \hat{N}).
\end{eqnarray*}
Using the above equations, it is straightforward to check that
\[
\begin{split}
(\nu^+_\phi)_z &= \frac{e^{v}}{2}(i - 2\theta e^{-2v})E_2^+(z) + \frac{e^{v}}{2}(1-2i\theta e^{-2v})E^+_3(z), \\
(\nu^-_\psi)_z &= \frac{e^{w}}{2}(i + 2\theta e^{-2w})E_2^-(z) - \frac{e^{w}}{2}(1+2i\theta e^{-2w})E^-_3(z),
\end{split}
\]
where $E^+_j(z),\,j = 2, 3$, is the oriented orthonormal reference of $\s^2_+$ at the point $\nu^+_{\phi}(z)$ given by
\[
E^{+}_2(z)=\frac{e^{-v}}{\sqrt 2}(\phi_x\wedge\phi+N\wedge\phi_y),\quad
E^+_3(z)=\frac{e^{-v}}{\sqrt 2}(\phi_x\wedge N+\phi_y\wedge\phi),
\]
and $E^-_j(z),\,j = 2, 3$, is the oriented orthonormal reference of $\s^2_-$ at the point $\nu^-_{\psi}(z)$ given by
\[
E^{-}_2(z)=\frac{e^{-w}}{\sqrt 2}(\psi_x\wedge\psi - \hat{N}\wedge\psi_y),\quad
E^-_3(z)=\frac{e^{-w}}{\sqrt 2}(\psi_x\wedge \hat{N} - \psi_y\wedge\psi).
\]
Therefore, taking into account the previous equations and that $JE^\pm_2 = E^\pm_3$, we obtain:
\begin{align*}
&\vprodesc{(\nu^+_{\phi})_z}{(\nu^+_{\phi})_z} = -2i\theta,& &\vprodesc{(\nu^-_{\psi})_z}{(\nu^-_{\psi})_z} =2i\theta,\\
&|(\nu^+_{\phi})_z|^2 =\frac{1}{2}(e^{2v}+4e^{-2v}|\theta|^2),& &|(\nu^-_{\psi})_z|^2=\frac{1}{2}(e^{2w}+4e^{-2w}|\theta|^2),\\
&\vprodesc{J(\nu^+_{\phi})_z}{(\nu^+_{\phi})_{\bar{z}}} = \frac{i}{2}(e^{2v}-4e^{-2v}|\theta|^2),& &\vprodesc{J(\nu^-_{\psi})_z}{(\nu^-_{\psi})_{\bar{z}}} =\frac{-i}{2}(e^{2w}-4e^{-2w}|\theta|^2).
\end{align*}

Hence
\[
\vprodesc{\nu_z}{\nu_z}=0,\quad\quad |\nu_z|^2=\frac{1}{2}(e^{2v}+e^{2w}+4|\theta|^2(e^{-2v}+e^{-2w})),
\]
where $\nu\equiv\nu_{\{\phi,\psi\}}$.

Now, from [RV], $\nu^+_{\phi}$ and $\nu^-_{\psi}$ are harmonic maps, and so $\nu$ is also a harmonic map. But previous equation says that $\nu$ is conformal, and so $\nu$ is a minimal immersion.

As $8|\theta|^2=e^{4v}|\sigma_{\phi}|^2=e^{4w}|\sigma_{\psi}|^2$, we easily obtain the expression of $g$. The computation of $C_j$ comes from the fact that $C_j = -2ie^{-2u} \vprodesc{J_j \nu_z}{\nu_{\bar{z}}}$ and the computations above. Finally, the previous equations also say that
\[
\langle J_1\nu_z,J_2\nu_z\rangle=\langle J(\nu_{\phi}^+)_z,J(\nu_{\phi}^+)_z\rangle-\langle J(\nu_{\psi}^-)_z,J(\nu_{\psi}^-)_z\rangle=-4i\theta,
\]
which finishes the proof.
\end{proof}

In the following result, the $1$-parameter family of minimal immersions in $\s^2 \times \R$ associated to a solution of the sinh-Gordon equation (cf.~Theorem~\ref{thm:sinh-Gordon->minima-s2xs2} and Remark~\ref{rmk:sinh-Gordon->minima-s2xs2}.\eqref{rmk:item:sinh-Gordon->minima-s2xR}) is explicitely described in terms of the corresponding $1$-parameter family of minimal immersions in $\s^3$, via the Gauss map (see the beginning of Section~\ref{sec:gauss-map}).

\begin{corollary}\label{cor:gauss-map-minimal-S2xR}
Let $v:\C \rightarrow \R$ be a solution of the sinh-Gordon equation $v_{z\bar{z}} + \frac{1}{2}\sinh(2v) = 0$. Then, the $1$-parameter family of isometric minimal immersions $\Phi_t: (\C,4\cosh^2v|dz|^2) \rightarrow \s^2 \times \R$ with Hopf differential $\Theta_t(z)=e^{it}dz\otimes dz$  is given by
\[
\Phi_t(z) = \left(\nu^+_{\phi_t}(z), 2\,\mathrm{Im}(z e^{it/2})\right),
\]
where $\phi_t: (\C,e^{2v}|dz|^2) \rightarrow \s^3$ is the $1$-parameter family of isometric minimal immersions with Hopf differentials $\Theta_{\phi_t}=\frac{i}{2}e^{it}dz\otimes dz$.
\end{corollary}

\begin{proof}
Let $w = 0$ be the trivial solution of the sinh-Gordon equation. The $1$-parameter family of flat minimal immersions $\psi_t: (\C,|dz|^2) \rightarrow \s^3$ associated to $w = 0$ and with Hopf $2$-differential $\Theta_{\psi_t} = \frac{i}{2}e^{it}dz \otimes dz$ is the deformation by isometries of the universal cover of the Clifford torus given by:
\[
\psi_t(z) = \frac{1}{\sqrt{2}}\left(e^{i\textrm{Re}[(1+i)ze^{it/2}]}, e^{i\textrm{Im}[(1+i)ze^{it/2}]}  \right),
\]
where $\s^3 = \{(u_1,u_2)\in \C^2:\, \abs{u_1}^2 + \abs{u_2}^2 = 1\}$. Then $\phi_t$ and $\psi_t$ are conformal immersions with the same Hopf differential and we can apply Theorem~\ref{thm:aplicacion-gauss-minima-s3} to conclude that $(\nu^+_{\phi_t}, \nu^-_{\psi_t}): \C \rightarrow \s^2 \times \s^2$ is a conformal minimal immersion. Its induced metric, in view of Theorem~\ref{thm:aplicacion-gauss-minima-s3}, is
\[
g = \frac{1}{2}\bigl( (2 + \abs{\sigma_\phi}^2)e^{2v}\abs{dz}^2 + 4\abs{dz}^2\bigr)  = 4\cosh^2(v)\abs{dz}^2,
\]
and its associated Hopf differential is $\Theta(z) = -2i\Theta_{\phi_t}(z) = -2i\Theta_{\psi_t}(z) = e^{it}dz\otimes dz$.

It is straighforward to check that
\[
\nu^-_{\psi_t}(z) = \cos [2\mathrm{Im}(ze^{it/2})]E_2^- + \sin[2\mathrm{Im}(ze^{it/2})]E_3^-,
\]
where $E_j^-$, $j = 1, 2, 3$, is the orthonormal reference in $\Lambda^2_-\R^4$ (see Section~\ref{sec:preliminares}) associated to the canonical basis in $\R^4$. Hence, we get that $\nu^-_{\psi_t} \subseteq \s^1 \subset (E^-_1)^\perp$ and, passing to the universal cover of $\s^1$, we get that $(\nu^+_{\phi_t}, \nu^-_{\psi_t})(z) = (\nu^+_{\phi_t}(z), 2\mathrm{Im}(ze^{it/2}))$. This finishes the proof.
\end{proof}

\begin{theorem}\label{thm:minima-s2xs2-localmente-aplicacion-gauss}
Let $\Phi:\Sigma\rightarrow\s^2\times\s^2$ be a minimal immersion of a simply-connected surface without complex points, i.e., $C_j^2<1$, $j=1,2$. Then $\Phi$ is congruent to the Gauss map of a pair of minimal immersions $\phi,\psi:\Sigma\rightarrow\s^3$ with conformal induced metrics and the same Hopf differentials.
\end{theorem}

\begin{proof}
Since the immersion $\Phi$ has no complex points, the Hopf differential has no zeroes. Hence, up to a change of complex coordinates if necessary, we can assume that $\Theta(z) = dz \otimes dz$. Now, from Proposition~\ref{prop:sinh-Gordon} and its proof that there exist two functions $v, w:\Sigma \rightarrow \R$ satisfying
\[
v_{z\bar{z}} + \frac{1}{2}\sinh(2v) = 0, \quad w_{z\bar{z}} + \frac{1}{2}\sinh(2w) = 0,
\]
such that $C_1 = \tanh(v -w)$, $C_2 = \tanh(v+w)$ and the conformal factor of the induced metric is $e^{2u} = 4\cosh(v+w)\cosh(v-w)$.

Following the begining of Section~\ref{sec:gauss-map}, there exist two one-parameter families of isometric minimal immersions $\phi_t: (\Sigma, e^{2v}|dz|^2) \rightarrow \s^3$ and $\psi_s: (\Sigma, e^{2w}|dz|^2) \rightarrow \s^3$, with Hopf differentials $\Theta_{\phi_{t}} = \frac{i}{2}e^{it}dz\otimes dz$ and $\Theta_{\psi_{s}} = \frac{i}{2}e^{is}dz\otimes dz$.

Now, we consider the immersions $\phi = \phi_{0}$ and $\psi = \psi_{0}$. Both minimal immersions have the same Hopf differentials $\Theta_\phi = \Theta_\psi = \frac{i}{2}dz\otimes dz$ and their induced metrics are conformal. Hence, we apply Theorem~\ref{thm:aplicacion-gauss-minima-s3} to obtain a minimal immersion $(\nu^+_\phi, \nu^-_\psi):\Sigma \rightarrow \s^2 \times \s^2$ which Hopf differential $\Theta = dz \otimes dz$.

Taking into account that $\abs{\sigma_\phi} = \sqrt{2}e^{-2v}$ and $\abs{\sigma_\psi} = \sqrt{2}e^{-2w}$ (see the proof of Theorem~\ref{thm:aplicacion-gauss-minima-s3}), we deduce, using Theorem~\ref{thm:aplicacion-gauss-minima-s3} again, that the Kähler functions and the induced metrics of the minimal immersions $\Phi$ and $(\nu^+_{\phi}, \nu^-_\psi)$ are the same. Finally, following the proof of Theorem~\ref{thm:sinh-Gordon->minima-s2xs2}, we also get that the fundamental data of $\Phi$ and $(\nu^+_{\phi}, \nu^-_\psi)$ are the same and so, both immersions are congruent.
\end{proof}

\section{Compact minimal surfaces}
\label{sec:compact}

This section is devoted to the study of compact minimal surfaces of $\s^2 \times \s^2$. The first result is a characterization of the slices as minimizing area surfaces among all the compact minimal ones.

\begin{proposition}\label{area}
Let $\Phi = (\Phi_1, \Phi_2):\Sigma\rightarrow\s^2\times\s^2$ be a minimal immersion of a compact surface $\Sigma$ with maximal multiplicity $\mu$. Then
\[
\hbox{Area}\,(\Sigma)\geq 4\pi\mu
\]
and the equality holds if and only if $\Phi$ is an embedding and $\Phi(\Sigma)$ is a slice of $\s^2\times\s^2$. In particular, $\text{Area}(\Sigma) \geq 4\pi$ and the equality is attained only by the slices.
\end{proposition}
\begin{proof}
Consider $\s^2\times\s^2\subset\R^6$ and let $\Psi:\Sigma\rightarrow\R^6$ be the corresponding immersion. It is clear that the maximal multiplicity of $\Phi$ and $\Psi$  are the same. If $H$ is the mean curvature of the immersion $\Psi$, from Li and Yau~\cite{LY} it follows that
\[
\int_{\Sigma}|H|^2dA\geq 4\pi\mu.
\]
If $\{e_1 ,e_2\}$ is an orthonormal reference on $\Sigma$ and $\bar{\sigma}$ is the second fundamental form of $\s^2 \times \s^2$ in $\R^6$, then
\[
2H = \sum_{j = 1}^2 \bar{\sigma}(e_j, e_j) = -\frac{1}{4}\sum_{j = 1}^2\bigl(|e_j - J_1J_2 e_j|^2\Phi_1,  |e_j + J_1J_2e_j|^2\Phi_2\bigr)
\]
and hence $8|H|^2=4+ (\hbox{Trace}\,A)^2$,
where $A$ is the matrix $A_{ij}=-\langle J_1e_i,J_2 e_j\rangle$. Let us observe that $|A_{ij}| \leq 1$ and so $(\hbox{Trace}\,A)^2\leq 4$ and the equality happens if and only if $A=\pm \mathrm{Id}$.

Joining both inequalities we deduce that $A(\Sigma) \geq 4\pi\mu$ and the equality holds if and only if $A = \pm \mathrm{Id}$. Finally, $A = \pm\mathrm{Id}$ means that $C_1^2 = C_2^2 = 1$ and hence $\Sigma$ is a slice.
\end{proof}

\begin{proposition}\label{prop:genus}
Let $\Phi = (\Phi_1, \Phi_2):\Sigma\rightarrow \s^2\times\s^2$ be a minimal immersion of an orientable compact surface $\Sigma$ of genus $g$, area $A$ and $\mathrm{degree}(\Phi_j) = d_j$.

\begin{enumerate}[(i)]
	\item If $\Phi$ is a complex immersion then $A = 4\pi (\abs{d_1} + \abs{d_2})$ and moreover, if $\Phi$ is also an embedding then $g = (\abs{d_1} - 1)(\abs{d_2} -1)$.

	\item If $\Phi$ is a complex immersion, then
	\begin{itemize}
		\item $A = 4\pi$ if and only if $g = 0$, $\Phi$ is an embbeding and $\Phi(\Sigma)$ is a slice.
		\item $A = 8\pi$ if and only if $g = 0$, $\Phi$ is an embedding and $\Phi$ is congruent to the graph of a biholomorphism of $\s^2$.
		\item $A = 12\pi$ if and only if $g = 0$, $\Phi$ is an embedding and $\Phi$ is congruent to the graph of an holomorphic map of $\s^2$ of degree $2$.
		\item $A = 16\pi$ and $\Phi$ is an embedding if and only if either $g = 0$ and $\Phi$ is congruent to the graph of a holomorphic map of $\s^2$ of degree $3$ or $g = 1$ and $\Phi$ is congruent to a Weierstrass torus.
	\end{itemize}

	\item If $g=0$, then $\Phi$ is a complex immersion. Moreover, $\Phi$ is an embedding if and only if $\Phi$ is congruent to the graph of an holomorphic map of $\s^2$ of degree $(A/4\pi)-1$.

	\item If $g = 1$ and $\Phi$ is not a complex immersion then $\Phi$ has no complex points and $d_1 = d_2 = \chi^\perp = 0$. If $g = 1$ and $\Phi$ is a complex immersion then $A = 16\pi$ if and only if $\Phi$ is congruent to a Weierstrass torus.
\end{enumerate}
\end{proposition}
\begin{remark}
From (iii) it follows that the real projective plane cannot be immersed in $\s^2\times\s^2$ as a minimal surface.
\end{remark}

\begin{proof}
(i) Taking into account~\eqref{eq:integral-formula-C} and that $\Phi$ is complex if $C_j^2 = 1$ for some $j \in \{1,2\}$, we deduce that $A = 4\pi(\abs{d_1} + \abs{d_2})$. Moreover, if $\Phi$ is an embedding, following the same reasoning of~\cite[Proposition 4]{CU} it follows that $\chi^{\perp}=2d_1d_2$. Also, from Lemma~\ref{lm:properties-fundamental-data}.\eqref{eq:modulo-f}, we obtain $K+(-1)^{j+1}K^\perp=1$ and hence integrating this equation we get
\[
4\pi(\abs{d_1} + \abs{d_2}) = A=4\pi(1-g)+(-1)^{j+1}2\pi\chi^{\perp} = 4\pi(1-g+(-1)^{j+1}d_1d_2).
\]
Therefore  $g = (\abs{d_1} -1)(\abs{d_2}-1)$.

(ii) Let assume that $\Phi$ is a complex immersion and, up to congruences, we can take $C_1 \equiv 1$ and so $d_1, d_2 \geq 0$. First, from Proposition~\ref{area}, $A=4\pi$ if and only if $\Phi$ is a slice. If $A\geq 8\pi$, then $\Phi(\Sigma)$ cannot be a slice and hence $d_1, d_2 \geq 1$.

If $A=8\pi = 4\pi(d_1 + d_2)$, then $d_1=d_2=1$, and hence $\Phi_1$ and $\Phi_2$ are diffeomorphisms of the $2$-sphere. We can reparametrize $\Phi$ by $\Psi=\Phi\circ\Phi_1^{-1}$, and hence $\Psi=(\mathrm{Id},\Phi_2\circ\Phi_1^{-1})$, i.e., $\Psi$ is the graph of a biholomorphism.

If $A=12\pi$, then either $d_1$ or $d_2$ is $1$ and so, making the same reasoning as above, we prove that $\Phi$ is the graph of a holomorphic map of $\s^2$ of degree $2$.

It is clear that the graph of a holomorphic map of $\s^2$ of degree $3$ and the Weierstrass tori are complex embeddings with $A = 16\pi$. Conversely, since $A = 16\pi$ and $\Phi$ is an embedding, from (i), we deduce that there are, up to congruencies, two posibilities:
\begin{itemize}
	\item $d_1 = 1$, $d_2 = 3$ and $g = 0$ and so, following a similar reasoning as before, $\Phi$ is congruent to the graph of a holomorphic map of $\s^2$ of degree $3$.
	\item $d_1 = d_2 = 2$ and $g = 1$. Hence, $\Phi_1, \Phi_2$ are elliptic functions of degree $2$ with disjoint branch points since $\Phi = (\Phi_1, \Phi_2)$ is an immersion. It is well known that, up to automorphism of the torus, such elliptic function is the Weierstrass $\wp$-function. Hence, $\Phi$ is congruent to one of the Weierstrass tori.
\end{itemize}

(iii) If $g=0$, the Riemann-Roch theorem says that the holomorphic Hopf differential $\Theta$ vanishes identically and so, from equation~\eqref{eq:modulo-diferencial-Hopf}, $\Phi$ is a complex immersion.

It is clear that the graph of a holomorphic map of $\s^2$ is an embedding. Conversely, if $\Phi$ is an embedding then, from (i), $0 = (d_1 -1)(d_2 - 1)$ and hence, either $d_1 = 1$ or $d_2 = 1$. Following the same argument as in (ii) we finish.

(iv) If $g=1$ and the immersion $\Phi$ is not complex, the Riemann-Roch theorem says that the holomorphic Hopf differential $\Theta$ has no zeroes. Hence, equation~\eqref{eq:modulo-diferencial-Hopf} ensures that $N^\pm_1 = N^\pm_2 = 0$ and so $d_1 = d_2 = \chi^\perp = 0$ from Lemma~\ref{lm:properties-fundamental-data}.\eqref{item:formulas-caracteristica-fibrado-normal}.

Now, let assume that $g = 1$ and $\Phi$ is a complex immersion. We already know that the Weierstrass tori have area $16\pi$. Conversely, let suppose that $\Sigma$ has area $16\pi$. We can assume, without loss of generality, that $C_1 = 1$. From (i), $4 = d_1 + d_2$ and so, since $g = 1$, it is follow that $d_1 = d_2 = 2$. The same argument used in (ii) finishes the proof.
\end{proof}

In the following theorem we characterize the slices and the surfaces $\mathbf{D}$ and $\mathbf{T}$ as the only ones with non-negative constant Gauss curvature. Moreover, we also characterize these examples by a gap-type theorem for the Gauss curvature.

\begin{theorem}\label{thm:curvaturaconstante}
Let $\Phi:\Sigma\rightarrow\s^2\times\s^2$ be a minimal immersion of a compact surface
$\Sigma$. Then:
\begin{enumerate}[(i)]
\item $\Sigma$ has positive constant curvature K if and only if either $K=1$ and $\Sigma$ is a slice or $K=1/2$ and $\Sigma$ is congruent to the diagonal  \textbf{D}.
\item $K=0$ if and only if $\Sigma$ is congruent to a finite covering of the flat torus \textbf{T}.

\item $K\geq 1/2$ if and only if either $K=1$ and $\Sigma$ is a slice or $K=1/2$ and $\Sigma$ is congruent to the diagonal  \textbf{D}.

\item $0 \leq K \leq 1/2$ if and only if either $K=0$ and $\Sigma$ is congruent to a finite covering of the flat torus $\mathbf{T}$ or $K=1/2$ and $\Sigma$ is congruent to the diagonal $\mathbf{D}$.
\end{enumerate}
\end{theorem}

\begin{proof}
(i)  We know that the slices have constant curvature $1$ and \textbf{D} has constant curvature $1/2$. Conversely, taking the two-fold oriented cover of $\Sigma$ we can assume that our surface is oriented, and the Gauss-Bonnet theorem says the genus of $\Sigma$ is $0$. From Proposition~\ref{prop:genus}.(iii), $\Phi$ is complex and so $K +(-1)^{j+1} K^\perp = 1$, for some $j = 1, 2$ (cf.\ Lemma~\ref{lm:properties-fundamental-data}.\eqref{eq:modulo-f}). Hence the normal curvature is also constant and Proposition~\ref{prop:totallygeodesic} finishes the proof.

(ii) We assume that $K=0$ and taking again the two-fold oriented cover of $\Sigma$ if necessary, we get that $\Sigma$ is a torus. From Proposition~\ref{prop:genus}.(iv), $\Phi$ is either a complex immersion or $C_1^2, C_2^2 < 1$. In the first case, $K^\perp$ is also constant and so Proposition~\ref{prop:totallygeodesic} gives a contradiction.  In the second case, from Lemma~\ref{lm:properties-fundamental-data}.\eqref{eq:laplacian-logarithm}, we obtain that $\Delta\log(1-C_j^2) = 2(-1)^{j+1}K^\perp$, $j = 1, 2$, and hence
\[
\Delta\log(1-C_1^2)(1-C_2^2)=0.
\]
As $\Sigma$ is compact, the armonic function $\log(1-C_1^2)(1-C_2^2)$ is constant and hence $(1-C_1^2)(1-C_2^2)=a$, $a\in\R^+$. It is clear that $a\leq 1$ and that $a=1$ if and only if $C_1=C_2\equiv 0$, but this means that our surface is $\mathbf{T}$. Hence, from now on we assume that $a<1$ and we will get a contradiction.

Deriving the above expression we get that
\[
C_1(1-C_2^2)\nabla C_1=-C_2(1-C_1^2)\nabla C_2,
\]
and from Lemma~\ref{lm:properties-fundamental-data}.\eqref{eq:laplacian-gradient-C} we get that
\[
(C_1^2-C_2^2)(1-a)=K^{\perp}(1-a-C_1^2C_2^2).
\]
Now, $1 - a - C_1^2C_2^2 = C_1^2(1-C_2^2) + C_2^2(1-C_1^2)$. Since we know that both $C_1$ and $C_2$ cannot vanish simultaneously in $p$ because $a < 1$ and the immersion $\Phi$ has not complex points, we deduce that $1-a-C_1^2C_2^2 > 0$. Using again the expression of $a$ we obtain that
\[
K^{\perp}=\frac{(1-a)(a-(1-C_1^2)^2)}{(1-a)(1-2C_1^2)+C_1^4}.
\]
Lemma~\ref{lm:properties-fundamental-data}.\eqref{eq:laplacian-gradient-C} says that the function $C_1$  is isoparametric. Let $U = \{p \in \Sigma \,|\, \nabla C_1(p) \neq 0\}$. Then, using~\cite[Lemma 3.3]{EGT} and that $K = 0$, a long but straightforward computations shows that $q(C_1)=0$ for a certain non-trivial polynomial $q$. That is, $C_1$ is constant in each connected component of $U$ and so $U = \emptyset$. Hence $C_1$ is constant and so $K^\perp$ is also constant. Finally, from Proposition~\ref{prop:totallygeodesic}, $\Sigma$ is locally $\mathbf{T}$ and so $C_1 = C_2 \equiv 0$, that is $a = 1$ which is a contradiction.

(iii) In this case, taking the oriented two-fold covering of $\Sigma$ if necessary, $\Sigma$ is again a sphere and so $\Phi$ is a complex immersion (cf.\ Proposition~\ref{prop:genus}.(iii)). Without loss of generality we can suppose that $C_1^2 = 1$. Using that $K+K^\perp=1$ (cf.\ Lemma~\ref{lm:properties-fundamental-data}.\eqref{eq:modulo-f}) in Lemma~\ref{lm:properties-fundamental-data}.\eqref{eq:laplacian-gradient-C} we obtain that
\[
\frac{1}{2}\Delta(1-C_2^2)=(2K-1)(1-C_2^2)+2C_2^2(C_2^2+1-2K).
\]
Since $C_2^2+1-2K\geq 0$ by Lemma~\ref{lm:properties-fundamental-data}.\eqref{eq:modulo-f} and $2K-1\geq 0$ by hypothesis we get that $\Delta(1-C_2^2) \geq 0$. Hence $C_2$ is constant and Proposition~\ref{prop:totallygeodesic} proves the result.

(iv) Since $K \geq 0$, taking the two-fold oriented cover of $\Sigma$ if necessary, from Gauss-Bonnet theorem, $\Sigma$ is either a torus or a sphere. In the first case, $K$ has to be constant zero and so, from (ii) the result follows. Hence, we suppose that $\Sigma$ is a sphere and, from Proposition~\ref{prop:genus}.(iii), the immersion $\Phi$ is complex. We can assume, without loss of generality, that $C_1 = 1$.

Now, we are going to get a Simon-type formula computing $\Delta |\sigma|^2$.
\begin{quote}
\noindent\textbf{Claim:} \emph{If $\Phi:\Sigma \rightarrow \s^2 \times \s^2$ is a complex immersion with respet to $J_1$ then:}
\[
\frac{1}{2}\Delta(|\sigma|^2 + C_2^2) = |\nabla \sigma|^2 + C_2^2(C_2^2 + K) + (1+K)(1-2K).
\]
\end{quote}
\noindent\textbf{Proof of the claim}. We can assume, without loss of generality, that $C_1 = 1$. Consider an oriented orthonormal frame $\{e_1, e_2\}$ in $T\Sigma$ satisfying $J_1 e_1 = e_2$ and so $J_1 \sigma(e_1, e_1) = \sigma(e_1, e_2)$. In the following we will denote $\sigma_{ij} = \sigma(e_i, e_j)$, $\nabla \sigma_{ijk} = (\nabla \sigma)(e_i, e_j, e_k)$, etc.

Using Codazzi equation (see Section~\ref{sec:preliminares}) and that $C_2 = \prodesc{J_2 e_1}{e_2}$ we get
\[
\frac{1}{2}\prodesc{\nabla|\sigma|^2}{e_k} = \sum_{ij} \prodesc{\nabla \sigma_{kij}}{\sigma_{ij}} = \sum_{ij} \prodesc{\nabla \sigma_{ijk}}{\sigma_{ij}} - \frac{1}{4}\prodesc{\nabla(C_2^2)}{e_k}.
\]
Therefore,
\[
\begin{split}
\frac{1}{2}\Delta\left(|\sigma|^2 + \frac{1}{2}C_2^2\right) &= \sum_{ijk} \left[ \prodesc{\nabla^2 \sigma_{kijk}}{\sigma_{ij}} + \prodesc{\nabla\sigma_{ijk}}{\nabla\sigma_{kij}} \right] = \\
&= \abs{\sigma}^2(1+K) + \sum_{ijk} \left[ \prodesc{\nabla^2 \sigma_{ikjk}}{\sigma_{ij}} + \prodesc{\nabla\sigma_{ijk}}{\nabla\sigma_{kij}} \right],
\end{split}
\]
where we have used Ricci identity in the second equality. Using again Codazzi equation 
\[
\sum_{k,j} \prodesc{\nabla \sigma_{kjk}}{\sigma_{ij}} = \sum_{k,j} \prodesc{\nabla \sigma_{jkk}}{\sigma_{ij}} - \frac{1}{2}C_2\prodesc{\nabla C_2}{e_i} = -\frac{1}{4} \prodesc{\nabla C_2^2}{e_i}.
\]
Deriving with respect to $e_i$ in this equations we get
\[
\sum_{ijk} \prodesc{\nabla^2\sigma_{ikjk} }{\sigma_{ij}} = -\sum_{ijk} \prodesc{\nabla\sigma_{kjk}}{\nabla\sigma_{iij}} - \frac{1}{4}\Delta(C_2^2),
\]
and finally we obtain
\[
\begin{split}
\frac{1}{2}\Delta(|\sigma|^2 + C_2^2) &= (1+K)|\sigma|^2 + \sum_{ijk}\bigl[ \prodesc{\nabla\sigma_{ijk}}{\nabla\sigma_{kij}} - \prodesc{\nabla\sigma_{kjk}}{\nabla\sigma_{iij}}\bigr] = \\
&= (1+K)|\sigma|^2 + \bigl[ |\nabla\sigma|^2 - \frac{1}{2}C_2^2(1-C_2^2)\bigr] - \bigl[\frac{1}{2}C_2^2(1-C_2^2)\bigr].
\end{split}
\]
where we have used again Codazzi equation in the second equality.
From Gauss equation we prove the claim.

Now, since $0 \leq K \leq 1/2$, we get that $\Delta(\abs{\sigma}^2 + C_2^2) \geq 0$ and hence $\abs{\sigma}^2 + C_2^2$ must be a constant function. Finally, from $\Delta(\abs{\sigma}^2 + C_2^2) = 0$ we deduce that $C_2 = 0$ and $K = 1/2$. Proposition~\ref{prop:totallygeodesic} finishes the proof.
\end{proof}

To finish this section, the following theorem states a rigidity result for the functions $K \pm K^\perp$.

\begin{theorem}\label{thm:curvatura-normal-desigualdad}
Let $\Phi:\Sigma\rightarrow\s^2\times\s^2$ be a minimal immersion of a compact orientable surface
$\Sigma$. Then:
\begin{enumerate}[(i)]
\item $K\pm K^\perp\geq 0$ if and only if either $\Phi$ is a complex immersion and so $K \pm K^\perp = 1$ or $\Phi$ is Lagrangian and so $K\pm K^\perp=0$.
\item $K\pm K^\perp<0$ can not happen.
\item If $\Sigma$ is a torus with $K^\perp = 0$ then $\Phi$ is non-full, i.e.,  $\Phi(\Sigma)$ is contained in a totally geodesic hypersurface of $\s^2 \times \s^2$.
\end{enumerate}
\end{theorem}

\begin{proof}
(i) We suppose that $K+(-1)^{j+1}K^\perp\geq 0$. Then, from Lemma~\ref{lm:properties-fundamental-data}.\eqref{eq:laplacian-gradient-C} we obtain that
\[
\frac{1}{2}\Delta(1-C_j^2)=(K+(-1)^{j+1}K^\perp)(1-C_j^2)+2C_j^2(C_j^2-(K+(-1)^{j+1}K^\perp)).
\]
Now, since $K + (-1)^{j+1}K^\perp \leq C_j^2$ from Lemma~\ref{lm:properties-fundamental-data}.\eqref{eq:modulo-f}, we get $\Delta(1-C_j^2)\geq 0$. Hence $C_j$ is a constant function and the previous equation ensures that
\[
(K+(-1)^{j+1}K^\perp)(1-C_j^2)=C_j^2(C_j^2-(K+(-1)^{j+1}K^\perp))=0.
\]
Now, there are only two possibilities:  $K+(-1)^{j+1}K^\perp=0$ and $C_j=0$ or $K+(-1)^{j+1}K^\perp=1$ and $C_j^2=1$.

Conversley, if $\Phi$ is a complex immersion then, from Lemma~\ref{lm:properties-fundamental-data}.\eqref{eq:modulo-f}, $K + (-1)^{j+1}K^\perp = 1$. If $\Phi$ is a Lagrangian immersion then, from Lemma~\ref{lm:properties-fundamental-data}.\eqref{eq:laplacian-gradient-C} it follows easily that $K + (-1)^{j+1}K^\perp = 0$.

(ii) We suppose that $K+(-1)^{j+1}K^\perp<0$ for some $j \in \{1, 2\}$. Firtly, $\Phi$ cannot be a complex immersion with respect to $J_j$ since in that case, from Lemma~\ref{lm:properties-fundamental-data}.\eqref{eq:modulo-f}, $K + (-1)^{j+1}K^\perp = 1$ which is a contradiction. Moreover, using again Lemma~\ref{lm:properties-fundamental-data}.\eqref{eq:laplacian-gradient-C}, any critical point $p$ of $C_j$ satisfy $C_j^2(p) = 1$ and hence $C_j$ is a non-constant function with only maxima and minima as critical points. The Morse theory says that $\Sigma$ is a sphere with $N_j^+ = N_j^-$. Hence, from Proposition~\ref{prop:genus}.(ii), $\Phi$ is a complex immersion with respect to $J_i$, $i \neq j$.

Now, from Lemma~\ref{lm:properties-fundamental-data}.\eqref{eq:modulo-f}, $K + (-1)^{i+1}K^\perp = 1$  and, since $K + (-1)^{j+1}K^\perp < 0$ we have that $K < 1/2$. Now, Gauss equation ensures that the second fundamental form $\sigma$ has no zeroes.

We are going to compute the Euler characteristic of the normal bundle. We will follow the notation of Section~\ref{sec:structure-equation}. First, since $f_i = 0$ (cf.~ Proposition~\ref{prop:fundamental-equations}), we get
\[
\begin{split}
\prodesc{\sigma(e_1, e_1)}{\tilde{N}} &= (-1)^{i+1}\prodesc{\sigma(e_1, e_2)}{N} = a, \\
\prodesc{\sigma(e_1, e_1)}{N} &= (-1)^{i}\prodesc{\sigma(e_1, e_2)}{\tilde{N}} = b,
\end{split}
\]
where $e_1 = e^{-u}\partial_x$ and $e_2 = e^{-u}\partial_y$. Let $p \in \Sigma$ and $v = \cos \theta e_1 + \sin \theta e_2$ a unit vector to $\Sigma$, then
\[
\sigma(v, v) = (a \cos 2\theta + b(-1)^i\sin 2\theta)\tilde{N} + (b \cos 2\theta + a(-1)^{i+1}\sin 2\theta)N,
\]
and hence $|\sigma(v,v)|^2 = a^2 + b^2 = \frac{1}{4}|\sigma|^2 > 0$. Therefore the well-defined map $F_p: \s^1 \subset T_p\Sigma \rightarrow (\s^1)^\perp \subset T_p^\perp \Sigma$ given by $F_p(v) = \frac{\sigma(v,v)}{|\sigma(v,v)|}$ has degree $2(-1)^{i+1}$ for any $p \in \Sigma$, because $\{\tilde{N}, N\}$ is an oriented frame on $T^\perp \Sigma$. Finally, we can apply~\cite[Theorem 1]{AFR} to conclude that $\chi^\perp = \deg(F)\chi = 4(-1)^{i+1}$.

Since $K + (-1)^{i+1}K^\perp = 1$ and $\chi = 2$ then $A = 4\pi + (-1)^{i+1}2\pi\chi^\perp = 12\pi$. On the other hand, since $N_j^+ = N_j^-$, Lemma~\ref{lm:properties-fundamental-data}.\eqref{item:formulas-caracteristica-fibrado-normal} affirms that $d_1 + (-1)^{j+1}d_2 = 0$ and so $A = 4\pi(\abs{d_1} + \abs{d_2}) = 8\pi \abs{d_1}$ (cf.\ Proposition~\ref{prop:genus}.(i)). Both expresions for the area lead to a contradiction.

(iii) Let supposet that $\Sigma$ is a torus and $K^\perp =0$. From Proposition~\ref{prop:genus}.(iv), either $\Phi$ is a complex immersion or $C_1^2, C_2^2 < 1$. The first case cannot happen since if $\Phi$ is complex then $1 = K \pm K^\perp = K$ (cf.\ Lemma~\ref{lm:properties-fundamental-data}.\eqref{eq:modulo-f}). Hence $C_1^2, C_2^2 < 1$. Now, from Lemma~\ref{lm:properties-fundamental-data}.\eqref{eq:laplacian-logarithm},
\[
\Delta \log\frac{1-C_1^2}{1-C_2^2} = 4K^\perp = 0.
\]
Hence the harmonic function $\log\frac{1-C_1^2}{1-C_2^2}$ is constant and so $(1-C_1^2) = a(1-C_2^2)$, $a \in \R^+$.

Now, $C_1^2(p) = C_2^2(p)$ at a point $p \in \Sigma$ if and only if $a = 1$. In this case, we get $C_1^2 \equiv C_2^2$ and Proposition~\ref{prop:curvaturanormalnula} proves the result. If $a \neq 1$ then $C_1^2 - C_2^2$ is either a positive or a negative function. On the other hand, using Lemma~\ref{lm:properties-fundamental-data}.\eqref{eq:laplacian-gradient-C} and the Gauss equation, we get
\[
\Delta (C_1^2 - C_2^2) = 2(C_1^2 - C_2^2)\bigl(3K - 2(C_1^2 + C_2^2) \bigr) = -(C_1^2 - C_2^2)(3|\sigma|^2 + C_1^2 + C_2^2),
\]
where $\sigma$ is the second fundamental form of $\Phi$. Hence $C_1^2 - C_2^2$ must be constant and from $\Delta(C_1^2 - C_2^2) = 0$ we deduce that $3|\sigma|^2 + C_1^2 + C_2^2 = 0$, that is, $C_1 = C_2 = 0$ which is a contradiction since $C_1^2 \neq C_2^2$.
\end{proof}

\end{document}